\documentclass[twoside,a4paper,11 pt]{amsart}

\usepackage[latin1]{inputenc}
\usepackage[T1]{fontenc}
\usepackage{amsmath,amssymb,latexsym,array}
\usepackage{graphicx}
\usepackage{amsxtra}
\usepackage[all]{xy}
\usepackage[lite]{amsrefs}
\usepackage{amsthm}
\usepackage{multicol}
\usepackage[small,bf]{caption}
\usepackage{wrapfig}

\usepackage{graphicx}

\newcommand{\Hy}{\mathcal{H}}
\newcommand{\C}{{\mathbb{C}}}
\newcommand{\R}{{\mathbb{R}}}
\newcommand{\Z}{{\mathbb{Z}}}
\newcommand{\N}{{\mathbb{N}}}
\newcommand{\F}{{\mathbb{F}}}
\newcommand{\rationals}{{\mathbb{Q}}}
\newcommand{\Ball}{\mathcal{B}}
\newcommand{\id}{\mathrm{id}}
\newcommand*{\Homol}{\operatorname{H}}
\newcommand*{\SL}{\operatorname{SL}}
\newcommand*{\PSL}{\operatorname{PSL}}

\renewcommand{\le}{\leqslant}
\renewcommand{\ge}{\geqslant}

\renewcommand{\geq}{\geqslant}

\theoremstyle{plain}
\newtheorem{thm}{\bfseries Theorem}
\newtheorem{Lem}[thm]{\bfseries Lemma}

\newtheorem{prop}[thm]{\bfseries Proposition}
\newtheorem{cor}[thm]{\bfseries Corollary}

\newtheorem{df}[thm]{\bfseries Definition}

\theoremstyle{remark}

\newtheorem{obs}[thm]{\bfseries Observation}

\newtheorem{rem}[thm]{\bfseries Remark}

\title[Integral homology of $\PSL_2$ of imaginary quadratic integers]{The integral homology of $\PSL_2$ of imaginary quadratic integers with non-trivial class group}
\author[Rahm]{Alexander D. Rahm}
\thanks{The first named author is supported by a research grant of the
Minist\`ere de l'Enseignement Sup\'erieur et de la Recherche, and partially supported by DFH-UFA grant CT-26-07-I and by a \mbox{DAAD (German Academic Exchange Service)} grant. \\
	The second named author was supported by a GK Gruppen und Geometrie \mbox{post-doctoral} fellowship of the DFG}
\address{Institut Fourier, UJF Grenoble and Math. Institut, Universit\"at G\"ottingen}
\email{Alexander.Rahm@ujf-grenoble.fr}
\urladdr{http://www-fourier.ujf-grenoble.fr/~rahm/}
\author[Fuchs]{Mathias Fuchs}
\address{Department of Bioinformatics, Center of Informatics, Statistics and Epidemiology UMG, University of G\"ottingen}
\email{mfu@bioinf.med.uni-goettingen.de}
\urladdr{http://www.bioinf.med.uni-goettingen.de/people/fuchs/}
\date{\today}
\begin{document}
\begin{abstract}
We show that a cellular complex defined by Fl\"oge allows to determine the integral homology of the \emph{Bianchi groups} $\text{PSL}_2 (\mathcal{O}_{-m})$, where $\mathcal{O}_{-m}$ is the ring of integers of 
an imaginary quadratic number field $\rationals\left[\sqrt{-m}\thinspace\right]$ for a square-free natural number $m$.\\
In the cases of non-trivial class group, we handle the difficulties arising from the cusps associated to the non-trivial ideal classes of $O_{-m}$. We use this to compute in the cases $m = 5, 6, 10, 13$ and $15$ the integral homology of $\text{PSL}_2(\mathcal{O}_{-m})$, which before was known only in the cases $m = 1, 2, 3, 7$ and $11$ with trivial class group.
\end{abstract}
\maketitle

\setcounter{secnumdepth}{3}
\setcounter{tocdepth}{2}
\tableofcontents

\section{Introduction}
The objects of study of this paper are the $\PSL_2$-groups $\Gamma$ of the ring of integers $\mathcal{O}_{-m}:= \mathcal{O}_{\rationals[\sqrt{-m}\thinspace]}$ of an imaginary quadratic number field $\rationals[\sqrt{-m}\thinspace]$, where $m$ is a square-free positive integer. 
We have $\mathcal{O}_{-m} = \Z[\omega]$ with $\omega = \sqrt{-m}$ for $m$ congruent to 1 or 2 modulo 4, and $\omega = -\frac{1}{2}+\frac{1}{2}\sqrt{-m}$ for $m$ congruent to 3 modulo 4. 

\begin{figure} 
\caption{Results in the group homology with simple integer coefficients} \label{resultstable}
$$
\begin{array}{llll}
&\Homol_q(\text{PSL}_2(\mathcal{O}_{-5}); \thinspace \Z)&\cong&
\begin{cases}
 \Z^2 \oplus \Z/3 \oplus (\Z/2)^2, \qquad  &\quad\thinspace\thinspace q = 1,\\
 \Z \oplus \Z/4 \oplus \Z/3 \oplus \Z/2, \qquad &\quad\thinspace\thinspace q=2,\\ 
 \Z/3 \oplus (\Z/2)^q, \qquad &\quad\thinspace\thinspace q\ge 3; 
\end{cases}
\vspace{.5em}
\\
&\Homol_q(\text{PSL}_2(\mathcal{O}_{-10}); \thinspace \Z)&\cong&
\begin{cases}
\Z^3  \oplus (\Z/2)^2, &\quad  q=1,\\
\Z^2 \oplus \Z/4 \oplus \Z/3 \oplus \Z/2, &\quad q=2,\\
\Z/3 \oplus (\Z/2)^{q}, & \quad q \ge 3;
\end{cases}
\vspace{.5em}
\\
&\Homol_q(\text{PSL}_2(\mathcal{O}_{-15}); \thinspace \Z)&\cong&
\begin{cases}
\Z^2 \oplus \Z/3 \oplus \Z/2, &\qquad\qquad q=1,\\
\Z \oplus  \Z/3 \oplus \Z/2, &\qquad\qquad q=2,\\
\Z/3 \oplus \Z/2, &\qquad\qquad  q \ge 3;
\end{cases}
\vspace{.5em}
\\
&\Homol_q(\text{PSL}_2(\mathcal{O}_{-13}); \thinspace \Z)&\cong&
\begin{cases}
\Z^3 \oplus (\Z/2)^2, &  q= 1, \\
\Z^2 \oplus \Z/4 \oplus (\Z/3)^2 \oplus \Z/2, & q = 2, \\
(\Z/2)^{q} \oplus (\Z/3)^2, & q = 4k+3, \quad k \ge 0, \\
(\Z/2)^q, & q = 4k+4,  \quad k \ge 0, \\
(\Z/2)^{q}, & q = 4k+1,  \quad k \ge 1, \\
(\Z/2)^{q} \oplus (\Z/3)^2, & q = 4k+2,  \quad k \ge 1;
\end{cases}
\vspace{.5em}
\\
&
\Homol_q(\text{PSL}_2(\mathcal{O}_{-6}); \thinspace \Z) &\cong &
\begin{cases}
\Z^2 \oplus \Z/3 \oplus \Z/2, & \thinspace \thinspace \thinspace q=1,\\
\Z \oplus \Z/4 \oplus \Z/3 \oplus (\Z/2)^2, & \thinspace \thinspace \thinspace q=2,\\
\Z/3 \oplus (\Z/2)^{2k+2}, & \thinspace \thinspace \thinspace q = 6k +3,\\
\Z/3 \oplus (\Z/2)^{2k+1}, & \thinspace \thinspace \thinspace  q = 6k+4,\\
\Z/3 \oplus (\Z/2)^{2k+4}, & \thinspace \thinspace \thinspace q = 6k +5,\\
\Z/3 \oplus (\Z/2)^{2k+3}, &  \thinspace \thinspace \thinspace q = 6k+6,\\
\Z/3 \oplus (\Z/2)^{2k+2}, & \thinspace \thinspace \thinspace q = 6k + 7,\\
\Z/3 \oplus (\Z/2)^{2k+5}, &  \thinspace \thinspace \thinspace q = 6k+8.
\end{cases}
\vspace{.5em}
\end{array}
$$
\end{figure}

The arithmetic groups under study have often been called Bianchi groups, because Luigi Bianchi \cite{Bianchi} computed fundamental domains for them as early as in 1892. 
They act on $\PSL_2(\C)$'s symmetric space, the hyperbolic three-space $\Hy$. 
Interest in this action first arose when Felix Klein and Henri Poincar\'e studied certain groups of M\"obius transformations with complex coefficients \cites{Klein, Poincare}, laying the groundwork for the study of \emph{Kleinian groups}. The latter are nowadays defined as discrete subgroups of $\PSL_2(\C)$.
Each non-cocompact arithmetic Kleinian group is commensurable with some Bianchi group \cite{Macl}. 
Thus, the Bianchi groups play a key role in the study of arithmetic Kleinian groups. 
A wealth of information on the Bianchi groups can be found in the pertinent monographs \cites{Fine, Elstrodt, Macl}.\\
Poincar\'e gave an explicit formula for their action on $\Hy$. However, the virtual cohomological dimension of arithmetic groups which are lattices in $\SL_2(\mathbb{C})$ is two, so it is desirable to restrict this proper action on $\Hy$ to a contractible cellular two-dimensional space.
Moreover, this space should be cofinite. 
In principle, this has been achieved by Mendoza \cite{Mendoza} and also by Fl\"oge \cite{Floege}, using reduction theory of Minkowski, Humbert, Harder and others. Their two approaches have in common that they consider two-dimensional \mbox{$\Gamma$-equivariant} retracts which are cocompact and are endowed with a natural CW-structure such that the action of $\Gamma$ is cellular and the quotient is a finite CW-complex.\\
Using Mendoza's complex, Schwermer and Vogtmann \cite{SchwermerVogtmann} calculated the integral group homology in the cases of trivial class group $m=1,2,3,7,11$, and Vogtmann \cite{Vogtmann} computed the rational homology as the homology of the quotient space in many cases of non-trivial class group. The integral cohomology in the cases $m=2,3,5,6,7,10,11$ has been determined by Berkove \cite{Berkove}, based on Fl\"oge's presentation of the groups with generators and relations. A completely different method to obtain group presentations has been chosen by Yasaki~\cite{Yasaki}, who has implemented an algorithm of Gunnells~\cite{Gunnells}
to compute the perfect forms modulo the action of $\mathrm{GL_2}(\mathcal{O}_{-m})$ 
and obtain the facets of the Vorono\"{\i} \thinspace polyhedron arising from a construction of Ash~\cite{Ash}. \\
It is the purpose of the present paper to show how Fl\"oge's complex can be used to obtain the integral homology of Bianchi groups also when the class group is non-trivial. 
We obtain the results displayed in figure \ref{resultstable}.
Thus for $q \ge 2$, the torsion in $\Homol_*(\text{PSL}_2(\mathcal{O}_{-5}); \thinspace \Z)$ is the same as that in $\Homol_*(\text{PSL}_2(\mathcal{O}_{-10}); \thinspace \Z)$, 
analogous to the cohomology results of Berkove \cite{Berkove}. 
The free part of these homology groups is in accordance with the rational homology results of Vogtmann \cite{Vogtmann}. \\
In the cases of non-trivial ideal class group, there is a difference between the approaches of Mendoza and Fl\"oge. 
We use the upper-half-space model of $\Hy$ and identify its boundary with $\C\cup\infty \cong \C P^1.$ 
The elements of the class group of the number field are in bijection with the $\Gamma$-orbits of the cusps, 
where the cusps are $\infty$ and the elements of the number field $\rationals\left[\sqrt{-d}\;\right]$, thought of as elements of the canonical boundary $\C P^1$. 
The cusps which represent a non-trivial element of the class group are commonly called \textit{singular} points. 
Whilst Mendoza retracts away from all cusps, Fl\"oge retracts away only from the non-singular ones. 
Rather than the space $\Hy$ itself, he considers the space $\widehat{\Hy}$ obtained from $\Hy$ by adjoining the $\Gamma$-orbits of the singular  points. We consider an analogous equivariant retraction of $\hat{\Hy}$ such that its retract $X$ contains the singular points. Now it turns out that the quotient space of $X$ by $\Gamma$ is compact, 
and $X$ is a suitable contractible $2$-dimensional $\Gamma$-complex also in the case of non-trivial class group.\\  
With an implementation in Pari/GP \cite{Pari}, due to the first named author, of Swan's algorithm \cite{Swan} we obtain a fundamental polyhedron for $\Gamma$ in $\Hy$. 
In the cases considered, Bianchi has already computed this polyhedron, so we have a control of the correctness of the implementation. \\
In the cases $m=5,6$ and $10$, Fl\"oge has computed the cell stabilizers and cell identifications; and with our Pari/GP program, we redo Fl\"oge's computations and do the same computation in the cases $m=13$ and $15$. We use the equivariant Euler characteristic to check our computations. 
Then we follow the lines of Schwermer and Vogtmann \cite{SchwermerVogtmann}, encountering a spectral sequence which degenerates on the $E^3$-page, in contrast to the cases of trivial class group where it does so already on the $E^2$-page. This is because of the singular points in our cell complex $X$, which have infinite stabilizers. So we have some additional use of homological algebra to obtain the homology of the Bianchi group. We give the full details for our homology computation in the case $m=13$. We then give slightly fewer details in the cases $m=5,6,10$ and $15$.\\
The authors would like to thank Philippe Elbaz-Vincent and Bill Allombert for many helpful discussions and hints on the techniques and the referee for helpful comments.\\
This article is dedicated to Fritz Grunewald (1949\thinspace-\thinspace2010).
\section{Fl\"oge's complex, contractibility and a spectral sequence}
Denote the hyperbolic three-space by $\Hy \cong \C\times\R^*_+$. 
We will not use its smooth structure, only its structure as a homogeneous $\SL_2(\C)$-space. 
The action is given by the formula
$$\begin{pmatrix}
a&b\\
c&d
\end{pmatrix} \cdot
(z,r):=\left(\dfrac{(\overline{d}-\overline{c}\overline{z})(az-b)-r^2\overline{c}a}{\left|cz-d\right|^2+r^2\left|c\right|^2},\dfrac{r}{\left|cz-d\right|^2+r^2\left|c\right|^2}\right);
$$
where $\left(\begin{smallmatrix}
a&b\\
c&d
\end{smallmatrix}\right)\in\SL_2(\C)$. 
As usual, we extend the action of $\SL_2(\C)$ to the boundary $\C P^1$ which we identify with $\{r=0\}\cup\infty \cong \C\cup\infty$. 
The action passes continuously to the boundary, where it reduces to the usual action by M\"obius transformations
$
\left(
\begin{smallmatrix}
a&b\\
c&d
\end{smallmatrix}
\right) \cdot
z=\frac{az-b}{-cz+d}.
$
As $-1\in\SL_2(\C)$ acts trivially, the action passes to $\PSL_2(\C)$. Now, fix a square-free $m\in\N$, let $\mathcal{O}_{-m}$ be the ring of integers in $\rationals[\sqrt{-m}\thinspace]$, and define $\Gamma=\PSL_2(\mathcal{O}_{-m})$. When the class number of $\rationals[\sqrt{-m}\thinspace]$ is one, 
then classical reduction theory provides a natural equivariant deformation retract of $\Hy$ which is a CW-complex. This complex is defined as follows. 
One first considers the union of all hemispheres 
$$S_{\mu,\lambda}:=\left\{(z,r): \quad \left|z-\dfrac{\lambda}{\mu}\right|^2+r^2=\dfrac{1}{\left|\mu\right|^2}\right\} \subset \Hy,$$
 for any two $\mu,\lambda$ with $ \mu \mathcal{O}_{-m}+\lambda \mathcal{O}_{-m}=\mathcal{O}_{-m}$. Then one considers the ``space above the hemispheres''
%{\setlength{\multlinegap}{0pt}
$$
B:=\bigl\{(z,r): \quad \left|cz-d\right|^2+r^2\left|c\right|^2\ge 1\\
 \text{ for all }c,d\in \mathcal{O}_{-m}, c \ne 0\text{ such that }c \thinspace \mathcal{O}_{-m} +d \mathcal{O}_{-m}=\mathcal{O}_{-m}\bigr\}$$
%} end of group with \multlinegap=0pt
and its boundary $\partial B$ inside $\Hy$. 
For nontrivial class group, the following definition comes to work.
\begin{df}
A point $s\in \C P^1 -\{\infty\}$ is called a \emph{singular point} if for all \mbox{$c,d\in \mathcal{O}_{-m}$,} 
$c\ne 0$, $c \thinspace \mathcal{O}_{-m} +d\mathcal{O}_{-m}=\mathcal{O}_{-m}$ we have $\left|cs-d\right|\ge 1$.
\end{df}
The singular points modulo the action of $\Gamma$ on  $\C P^1$  are 
in bijection with the nontrivial elements of the class group \cite{SerreA}.
In \cite{Floege}, Fl\"oge extends the hyperbolic space $\Hy$ 
to a larger space $\widehat{\Hy}$ as follows.
\begin{df}
  As a set, $\widehat{\Hy}\subset \C\times \R^{\ge 0}$ is the closure under the $\Gamma$-action of the union \\
$\widehat{B} := B\cup\{\text{singular points}\}$. 
The topology is generated by the topology of $\Hy$ together with the following neighborhoods of the translates $s$ of singular points:
$$
\widehat{U}_\epsilon(s):=\{s\}\cup \begin{pmatrix}
s&0\\-1&s^{-1}
\end{pmatrix} \cdot
\left\{(z,r)\in\Hy :r>\epsilon^{-1}\right\}.
$$
\end{df}
\begin{rem} The matrix $\left(\begin{smallmatrix}s&0\\-1&s^{-1}\end{smallmatrix}\right)$ maps the point at infinity into $s$, 
thus giving the point $s$ the topology of $\infty$. 
The neighborhood $\widehat{U}_\epsilon(s)$ is sometimes called a ``horoball'' because in the upper-half space model it is a Euclidean ball, 
but with the hyperbolic metric it has ``infinite radius''. 
\end{rem}
The space $\widehat{\Hy}$ is endowed with the natural $\Gamma$-action.
Now the essential aspect of Fl\"oge's construction is the following consequence of 
Fl\"oge's theorem \cite{FloegePhD}*{6.6}, which we append as theorem \ref{equivariant_retraction}.
\begin{cor}
There is a retraction $\rho$ from $\widehat{\Hy}$ onto the set $X \subset \widehat{\Hy}$ of all $\Gamma$-translates of $\partial \widehat{B}$, 
i.~e. there is a continuous map $\rho:\widehat{\Hy}\to X$ such that $\rho(p)=p$ for all $p\in X$.
The set $X$ admits a natural structure as a cellular complex $X^\bullet$ on which $\Gamma$ acts cellularly.
\end{cor}
\begin{rem}
\begin{enumerate}
\item
We show with the lemma below that $\rho$ is a homotopy equivalence, 
without giving a continuous path of maps $\widehat{\Hy}\to\widehat{\Hy}$ connecting $\rho$ to the identity on $\widehat{\Hy}$.
\item
The map $\rho$ is $\Gamma$-equivariant because its fibers are geodesics. But we do not make use of this fact, 
as we do not need to show that the homotopy type of $\Gamma\backslash \widehat{\Hy}$ is the same as that of $\Gamma\backslash X$. 
This would be useful in the case of trivial class group, i.\ e. the case of a proper action, 
to compute the rational homology $\Homol_*(\Gamma; \thinspace \rationals)\cong\Homol_*(\Gamma\backslash \Hy; \thinspace \rationals)$.
\item We will provide $X^\bullet$ with a cellular structure which is fine enough to make the cell stabilizers fix the cells pointwise. 
\end{enumerate}
\end{rem}
\begin{Lem}
Let $Y$ be a CW-complex which admits an inclusion $i$ into a contractible topological space $A$, 
such that $i$ is a homeomorphism between $Y$ with its cellular topology and the image $i(Y)$ with the subset topology of $A$. 
Let $p:A\to Y$ be a continuous map with $p \circ i = \id_Y$. Then $p$ is a homotopy equivalence.
\end{Lem}
\begin{proof}
For all $n \in \N$, the induced maps on the homotopy groups 
$(\id_Y)_* = (p \circ i)_* : \pi_n(Y) \to \pi_n(Y)$  factor through $\pi_n(A) = 0$,
hence are the zero map; and $\pi_n(Y) = 0$.
Denote by $c$ the constant map from $A$ to the one-point space. 
Then $c \circ i$ is a morphism of CW-complexes, and the zero maps it induces on the homotopy groups are isomorphisms.  
Thus by Whitehead's Theorem, $c \circ i$ is a homotopy equivalence. 
As $A$ is contractible, the composition $(c \circ i) \circ p = c$ is a homotopy equivalence,
so the same holds already for $p$.
\end{proof}
Taking $Y=X$, $A = \widehat{\Hy}$, $p = \rho$, and using lemma \ref{Extended space contractible}, we obtain a crucial fact for our computations.
\begin{cor}\label{contractibility}
$X^\bullet$ is contractible.
\end{cor}
The following is an observation on Fl\"oge's construction.
\begin{Lem} \label{Extended space contractible}
The space $\widehat{\Hy}$ is contractible.
\end{Lem}
\begin{proof}
One can identify the boundary of $\Hy \cong \{ (z,r) \in  \C \times \R \thinspace | \thinspace r>0 \}$ 
with $\C P^1 \cong \C\cup\infty \cong \{r=0\}\cup\infty$. 
By viewing the singular points as part of the boundary, we arrive at an upper half-space model of $\widehat{\Hy}$.\\
Now consider $\Hy_1:=\{(z,r)\in\widehat{\Hy}:r\ge 1\}$ with the subspace topology of $\widehat{\Hy}$. 
The idea of the proof is to consider a vertical retraction onto $\Hy_1$, and to show by an explicit argument that preimages of open sets are open. Fl\"oge \cite{FloegePhD}*{Korollar 5.8} suggests using the map
$$
\phi:\widehat{\Hy}\times[0,1]\to\widehat{\Hy},\ ((z,r),t)\mapsto \begin{cases}
(z,r)\text{ for all }t\in[0,1],&\text{if }r\ge 1\\
(z,r+t(1-r)),&\text{if }r<1.
\end{cases}
$$
Let us now check that this is a continuous family of continuous maps. 
Consider the collection of open balls with respect to the Euclidean metric on $\C\times\R_+$ as soon as they are either contained in $\C\times\R^*_+$, 
or touch the boundary $\C\times \{0\}$ in a cusp in $\widehat{\Hy}-\Hy$. 
This is a basis for the topology of $\widehat{\Hy}$. Consider one such open ball $\Ball$, and its preimage under some $\phi_t$, $t\in[0,1)$. 
This either lies entirely in $\Hy$, and is open, or it has boundary points. 
In the latter case, consider the inverse of $\phi_t$ on $\widehat{\Hy}-\Hy_1$, given by
$$
\phi_t^{-1}=\bigl(z,\tfrac{r-t}{1-t}\bigr)
$$ 
if this is in $\widehat{\Hy}$. Suppose there is a cusp $s$ with $s\in\widehat{\Hy}-\Hy$ and $\phi_t(s,0)\in \Ball$. As $\Ball$ is open, we find $\beta>0$ and $\delta>0$ such that $(s,t+\beta)$ and $(s+\delta,t)$ are in $\Ball$.
Since
$$
\begin{cases}
\phi_t\left(s,\frac{\beta}{1-t}\right)=(s,t+\beta)\in \Ball\\
\phi_t(s+\delta,0)=(s+\delta,t)\in \Ball,\\
\end{cases}
$$
we know that $(s,\frac{\beta}{1-t})$ and $(s+\delta,0)$ are in the preimage of $\Ball$ under $\phi_t$. We deduce that the whole horosphere of Euclidean diameter $\text{min }\{\beta,\delta\}$ touching at the cusp $s$ is included in the preimage of $\Ball$. Thus each point of the preimage has a neighborhood entirely contained in the preimage, which therefore also is open. The continuity at $t=1$ as well as the continuity in the variable $t$ follow from very similar arguments.
The space $\Hy_1$ is homeomorphic to $\C\times \R_+$, thus contractible.
\end{proof}

\subsection{The equivariant spectral sequence in group homology}$\mbox{ }$\label{SpecSeq}\\
Corollary \ref{contractibility} gives us a contractible complex $X^\bullet$ on which $\Gamma$ acts cellularly.
As a consequence, the integral homology $\Homol_*(\Gamma; \thinspace \Z)$ 
can be computed as the hyperhomology $\mathbb{H}_*(\Gamma; \thinspace  C_\bullet(X))$ 
of $\Gamma$ with coefficients in the cellular chain complex $C_\bullet(X)$ associated to $X$. 
This hyperhomology is computable because there is a spectral sequence as 
in \cite{Brown}*{VII} which is also the one used in \cite{SchwermerVogtmann}. 
It is the spectral sequence associated to the double complex $\Theta^\Gamma_\bullet\otimes_{\Z\Gamma} C_\bullet(X)$ 
computing the hyperhomology, 
where we denote by $\Theta^\Gamma_\bullet$ the bar resolution of the group 
$\Gamma$. This spectral sequence can be rewritten 
(see \cite{SchwermerVogtmann}*{1.1}) to yield
$$
E^1_{p,q}=\bigoplus_{\sigma\thinspace\in\thinspace\Gamma\backslash X^p}\Homol_q(\Gamma_\sigma; \thinspace \Z)\implies \Homol_{p+q}(\Gamma; \thinspace \Z),
$$
where $\Gamma_\sigma$ denotes the stabilizer of (the chosen representative for) the $p$-cell $\sigma$. 
We have stated the above $E^1$-term with trivial $\Z$-coefficients in $\Homol_q(\Gamma_\sigma; \thinspace \Z)$,
 because we use a fundamental domain which is strict enough to give $X$ a cell structure on which $\Gamma$ acts without inversion of cells. 
We shall also make extensive use of the description of the $d^1$-differential given in \cite{SchwermerVogtmann}. 
\\
The technical difference to the cases of trivial class group, treated by \cite{SchwermerVogtmann}, 
is that the stabilizers of the singular points are free abelian groups of rank two. 
In particular, the $\Gamma$-action on our complex $X^\bullet$ is not a \textit{proper action} in the sense that all stabilizers are finite. As a consequence, the resulting spectral sequence does not degenerate on the $E^2$-level as it does in Schwermer and Vogtmann's cases.\\
So we compute a nontrivial differential $d^2$, making some additional use of homological algebra, in particular the below lemma and its corollary.
\begin{rem}\label{remark}
It would be possible to shift the technical difficulty away from homological algebra, using a topological modification of our complex.
In our cases of class number two, there is one singular point in the fundamental domain, representing the nontrivial element of the class group.
Its stabilizer is free abelian of rank two, and contributes the homology of a torus to the zeroth column of the $E^2$-term of our spectral sequence:
$\Homol_1(\Z^2; \thinspace \Z) \cong \Z^2$, $\Homol_2(\Z^2; \thinspace \Z) \cong \Z$ and $\Homol_q(\Z^2; \thinspace \Z) = 0$ for $q>2$.  
One could modify our complex in order to make the $\Gamma$-action on it proper, by replacing each singular point by an $\R^2$ with the former stabilizer $\Z^2$ now acting properly. Then the nontriviality of our differential is equivalent to the existence of a nontrivial homology relation induced by adjoining the torus $\R^2/\Z^2$ to the fundamental domain.
\end{rem}
%\subsubsection{The $d^2$-differential}$\text{ }$\\
The following lemma will be useful for computing our $d^2$-differential in the situations where cycles for $\Gamma_\sigma$ are given in terms of the standard resolution of $\Gamma$ instead of $\Gamma_\sigma$. In order to state it, let $\Gamma_\sigma$ be a finite subgroup of $\Gamma$, let $M$ be a $\Z\Gamma_\sigma$-module, and $\ell:\Gamma/\Gamma_\sigma\to\Gamma$ a set-theoretical section of the quotient map $\pi:\Gamma\to\Gamma/\Gamma_\sigma$. Further, denote the standard bar resolution of a discrete group $\Gamma$ by $\Theta_\bullet^\Gamma$. It will be convenient to view $\Theta_\bullet^\Gamma$ as a complex of $\Z\Gamma$-right modules resp. $\Z\Gamma_\sigma$-right modules. Thus, $\Theta_q^\Gamma$ is defined as the free $\Z$-module generated by the $(q+1)$-tuples $(\gamma_0,\dots,\gamma_q)$ of elements of $\Gamma$ with the action given by $(\gamma_0,\dots,\gamma_q).\gamma = (\gamma_0\gamma,\dots,\gamma_q\gamma)$ and the same boundary operator as in the left module case, namely $\partial =\sum_{i=0}^q (-1)^i d_i$ where $d_i(\gamma_0,\dots,\gamma_q)=(\gamma_0,\dots, \widehat{\gamma_i},\dots,\gamma_q)$.
\begin{Lem}\label{Lemma}
The section $\ell$ defines a map of $\Z\Gamma_\sigma$-complexes
$$
\hat{\varepsilon_\ell}:\Theta_\bullet^\Gamma\longrightarrow\Theta_\bullet^{\Gamma_\sigma}
$$
of degree zero which is a retraction of the resolution $\Theta_\bullet^\Gamma$ of the group $\Gamma$ to the resolution $\Theta_\bullet^{\Gamma_\sigma}$ of $\Gamma_\sigma$. For each $\gamma\in\Gamma$, $\ell(\pi(\gamma))$ is in the same orbit of the $\Gamma_\sigma$-right-action on $\Gamma$ as $\gamma$, so $(\ell(\pi(\gamma)))^{-1} \gamma \in \Gamma_\sigma$. The map $\hat{\varepsilon_\ell}$ is induced on $\Theta_0^ \Gamma=\Z\Gamma$ by
\begin{align*}
 \Gamma  \xrightarrow{\ \varepsilon_\ell\ } \Gamma_\sigma\to\Z\Gamma_\sigma,\\
\quad \quad \gamma \mapsto  (\ell(\pi(\gamma)))^{-1} \gamma
\end{align*}
and is continued as a tensor product 
$\hat{\varepsilon_\ell} = \varepsilon_\ell \otimes ... \otimes \varepsilon_\ell 
= \varepsilon_\ell^{\otimes (n+1)}$ on $\Theta_n^ \Gamma$.
\end{Lem}
\begin{rem}
\begin{enumerate}
\item
Since the group $\Gamma_\sigma$ acts from the right, the map $\varepsilon_\ell$ is $\Z\Gamma_\sigma$-linear.
\item
Note that the resulting isomorphism in homology from
$
\Homol_*(\Theta_\bullet^\Gamma\otimes_{\Z\Gamma_\sigma}M)
$ to $
\Homol_*(\Theta_\bullet^{\Gamma_\sigma}\otimes_{\Z\Gamma_\sigma}M)
$
is independent of the choice of $\ell$, and consistent with 
the canonical isomorphisms of both sides with $\Homol_*(\Gamma_\sigma; \thinspace M)$.
\item
Note that in the above lemma,  it is not necessary to require \mbox{$\ell(\pi(1))=1$.} 
This would imply that $\varepsilon_\ell$ is the identity on 
$\Theta_\bullet^{\Gamma_\sigma}$. 
However, we will choose $\ell(\pi(1))=1$ for simplicity.
\item
In explicit terms, the map $\varepsilon_\ell$ is described as follows:
$$ \varepsilon_\ell: {\mathbb{Z}} \Gamma 
\rightarrow {\mathbb{Z}} \Gamma_\sigma,$$
$$ \sum_{\gamma \in \Gamma} a_\gamma \gamma 
= \sum_{\gamma_\sigma \in \Gamma_\sigma} 
\sum_{\rho \in \Gamma / \Gamma_\sigma} 
a_{\gamma_\sigma  \ell(\rho)} \gamma_\sigma  \ell(\rho)
\mapsto \sum_{\gamma_\sigma \in \Gamma_\sigma} 
\biggl(\sum_{\rho \in \Gamma / \Gamma_\sigma} 
a_{\gamma_\sigma  \ell(\rho)}\biggr)\gamma_\sigma,   $$
where the $a_\gamma$ are coefficients from ${\mathbb{Z}}$.
The map $ \varepsilon_\ell$ restricts to the identity on 
${\mathbb{Z}} \Gamma_\sigma$ and gives an isomorphism of 
${\mathbb{Z}}$-modules from
$ {\mathbb{Z}} [\ell(\rho) \Gamma_\sigma] 
$ to $ {\mathbb{Z}} \Gamma_\sigma$
for every $\Gamma_\sigma$-orbit $\ell(\rho) \Gamma_\sigma$.
\end{enumerate}
\end{rem}
\begin{proof}[Proof (of the lemma).]
In fact, the statement holds for any chain map $\hat{\varepsilon}$ 
in the place of $\hat{\varepsilon_\ell}$ that satisfies the following conditions. 
They are easily checked to hold for the maps $\hat{\varepsilon_\ell}$.
\begin{enumerate}
\item
$\hat{\varepsilon}$ is $\Z\Gamma_\sigma$-linear.
% \item
% $\hat{\varepsilon}$ is the identity on the factors $\Z\Gamma_\sigma\subset \Z\Gamma$ in each degree.
\item
The augmentation $\Theta_0^\Gamma\to\Z$ is the composition of 
$\hat{\varepsilon}$ with the augmentation $\Theta_0^{\Gamma_\sigma}\to\Z$.
\end{enumerate}
Then the statement follows from the comparison theorem \cite{Weibel}*{2.2.6} 
of fundamental homological algebra. In fact, the properties imply that 
$\hat{\varepsilon}$ is a chain map of resolutions lifting the identity on $\Z$.
 An inverse is given by the canonical inclusion 
 $\Theta_\bullet^{\Gamma_\sigma}\to\Theta_\bullet^\Gamma$, 
 and since the composition is unique up to chain homotopy equivalence, 
 it must be homotopic to the identity.
\end{proof}
The group $\Gamma_\sigma$ acts diagonally from the right on 
$\Theta_1^ \Gamma \cong \Z\Gamma \otimes_\Z \Z\Gamma$, and trivially on $\Z$,
so we can consider $\Theta_1^ \Gamma \otimes_{\Z\Gamma_\sigma} \Z$.  Denote the commutator quotient map $\Gamma_\sigma \to (\Gamma_\sigma)^\mathrm{ab} \cong \Homol_1(\Gamma_\sigma)$ by $a\mapsto \overline{a}$.
\begin{cor} \label{Corollary}
 Consider a cycle for $\Homol_1(\Gamma_\sigma; \thinspace  \Z )$ of the form $\sum_i (a_i\otimes_\Z b_i) \otimes_{\Z\Gamma_\sigma} 1 \in \Theta_1^\Gamma\otimes_{\Z\Gamma_\sigma}\Z$ where $a_i,b_i\in\Z\Gamma$. Assume that all $a_i$ and $b_i$ are elements of $\Gamma$. The ensuing homology class is then given by 
$$\sum_i \overline{   \varepsilon_\ell(b_{i})  \varepsilon_\ell(a_i)^{-1} }  \in (\Gamma_\sigma)^\mathrm{ab}.$$
\end{cor}
By the linearity of the described map, this covers the general case as the cycles of the form appearing in the corollary generate the submodule of all cycles. Note that the cycle condition on $\sum\limits_{i} (a_i\otimes_\Z b_i) \otimes_{\Z\Gamma_\sigma} 1$ says that $\sum\limits_i (b_i-a_i) \otimes_{\Z\Gamma_\sigma} 1=0,$ which means that $\sum_i a_i$ is equivalent to $\sum_i b_i$ modulo $\Z\Gamma_\sigma$.
\begin{proof}[Proof (of the corollary).]
Using lemma \ref{Lemma}, we apply the map
$$ (\varepsilon_\ell \otimes_{\mathbb{Z}}  \varepsilon_\ell) 
	\otimes_{{\mathbb{Z}} \Gamma_\sigma }  1 : 
({\mathbb{Z}} \Gamma \otimes_{\mathbb{Z}} {\mathbb{Z}} \Gamma )
	\otimes_{{\mathbb{Z}} \Gamma_\sigma }  {\mathbb{Z}}
\longrightarrow 
({\mathbb{Z}} \Gamma_\sigma \otimes_{\mathbb{Z}}  {\mathbb{Z}} \Gamma_\sigma )
 	\otimes_{{\mathbb{Z}} \Gamma_\sigma }  {\mathbb{Z}}
$$
to get
\begin{align*}
\sum(\varepsilon_\ell \otimes_{\mathbb{Z}} 
 \varepsilon_\ell \otimes_{\Z\Gamma_\sigma} 1)( a_i\otimes_\Z b_i \otimes_{\Z\Gamma_\sigma} 1)
= \sum
(\varepsilon_\ell( a_i) \otimes_\Z \varepsilon_\ell(b_i)) \otimes_{\Z\Gamma_\sigma} 1.
\end{align*}
Denote the augmentation $\Z\Gamma_\sigma\to\Z$ by $\varepsilon$. As $a_i \in \Gamma$, we have $\varepsilon_\ell( a_i) \in \Gamma_\sigma$ which is invertible in $\Z\Gamma_\sigma$, and $\varepsilon(\varepsilon_\ell( a_i)) = 1$. So the above term equals
$$ \sum
\left( 1 \otimes_\Z \varepsilon_\ell(b_i)(\varepsilon_\ell( a_i))^{-1}\right) \otimes_{\Z\Gamma_\sigma} \varepsilon( \varepsilon_\ell(a_i))\\
= \sum
\left( 1 \otimes_\Z \varepsilon_\ell(b_i)(\varepsilon_\ell( a_i))^{-1}\right) \otimes_{\Z\Gamma_\sigma} 1,$$
where we take into account that the action of $\Z\Gamma_\sigma$ on ${\mathbb{Z}} \Gamma_\sigma \otimes_{\mathbb{Z}}  {\mathbb{Z}} \Gamma_\sigma$ is the diagonal right action, and that of $\Z\Gamma_\sigma$ on $\Z$ is the trivial action $a\cdot 1 = \varepsilon(a)$ for $a\in\Z\Gamma_\sigma$. In bar notation, we thus obtain the cycle $\sum \left[   \varepsilon_\ell(b_i)(\varepsilon_\ell( a_i))^{-1}\right] 
\otimes_{\Z\Gamma_\sigma} 1$,  which is mapped to
$$
\sum_i \overline{   \varepsilon_\ell(b_{i})  \varepsilon_\ell(a_i)^{-1} } 
\in  (\Gamma_\sigma)^\mathrm{ab}
$$
by the map described in \cite{Brown}*{page 36}; it is easy to check that an isomorphism  $\Homol_1(\Theta_\bullet^G\otimes_G \Z)\cong G^\mathrm{ab}$ is described by $(1\otimes g)\otimes_G 1=\left[g\right]\otimes_G 1\mapsto \overline{g}$ also in the case where $\Theta^G_\bullet$ is acted on by $G$ from the right. Moreover, this isomorphism is natural with respect to group inclusions.
\end{proof}

Another property of the spectral sequence is that a part of it can be checked whenever the geometry of the fundamental domain and a presentation of $\Gamma$ are known. As Fl\"oge shows, an inspection of the complex $X$ and the associated stabilizer groups and identifications yields, together with \mbox{\cite{Abels}*{theorem 4.5},} a presentation of $\Gamma$ by means of generators and relations. We will use the presentation computed by Fl\"oge for $m=5,6,10$ and that computed by Swan \cite{Swan} for $m=15$.
\begin{rem}\label{check}
Let us look at the low term short exact sequence
$$ \xymatrix{
0 \ar[r] & E^\infty_{0,1} \ar[r] & \Gamma^{\text{ab}} \ar[r]^\rho & E^\infty_{1,0} \ar[r] & 0
}$$
of the spectral sequence. We have $E^\infty_{1,0}=\Homol_1(\Gamma\backslash X)=(\pi_1(\Gamma\backslash X))^{\text{ab}}$, and the projection $\rho$ is the abelianization of the map $\Gamma\to\pi_1(\Gamma\backslash X)$ given as follows. Choose a fixed base point $x\in X$. For every $\gamma\in\Gamma$, choose a continuous path in $X$ from $x$ to $\gamma x$. This gives a well-defined loop in $\Gamma\backslash X$ since $X$ is contractible.\\
The abelianization of $\Gamma$ can be immediately deduced from its presentation. Thus, we can compute the group $E_{0,1}^\infty=E_{0,1}^3$ as the kernel of the projection $\rho$ and check this with the result obtained from detailed analysis of the $d^2$-differential.
\end{rem}

\subsection{The homology of the finite subgroups in the Bianchi groups}$\mbox{ }$\\
In order to compute the $E^1$-term of the spectral sequence introduced in section \ref{SpecSeq}, we will need the isomorphism classes of the homology groups of the stabilizers.
\begin{Lem}[Schwermer/Vogtmann \cite{SchwermerVogtmann}] \label{finiteSubgroups}
The only isomorphism classes of finite subgroups in $\mathrm{PSL}_2({\mathcal{O}})$ are the cyclic groups of orders two and three, the trivial group, the Klein four-group ${\mathcal{D}}_2 \cong \Z/2 \times \Z/2$, the symmetric group~${\mathcal{S}}_3$ and the alternating group ${\mathcal{A}}_4$. \\
The homology with trivial $\Z$ respectively $\Z/n$-coefficients, $n = 2$ or $3$, of these groups is
\tiny
\begin{alignat*}6
&
\Homol_q(\Z/n; \thinspace  \thinspace \Z)
&
 \cong
 & 
  \begin{cases}
   \Z, & q = 0, \\
   \Z/n, & q \medspace \mathrm{ odd}, \\
   0, & q \medspace \mathrm{ even}, \medspace q > 0;
   \end{cases}
&
\quad
& 
 \Homol_q(\Z/n; \thinspace \Z/n)
&\cong
&\thinspace \Z/n, q \in \N \cup \{0\};
&\quad
  &
&
&
\\
& \Homol_q({\mathcal{D}}_2; \thinspace \Z)
 &
 \cong 
 &
\begin{cases}
 \Z, & q = 0, \\
 (\Z/2)^\frac{q+3}{2}, & q \medspace \mathrm{ odd}, \\
 (\Z/2)^\frac{q}{2}, & q \medspace \mathrm{ even}, \medspace q > 0;
 \end{cases}
&\quad&
\Homol_q({\mathcal{D}}_2; \thinspace \Z/2) &
\cong
& (\Z/2)^{q+1}
&\quad&
   \Homol_q({\mathcal{D}}_2; \thinspace \Z/3) &=& 0, q \geq 1;\\
&  \Homol_q({\mathcal{S}}_3; \thinspace \Z) &\cong& 
\begin{cases}
 \Z, & q = 0, \\
 \Z/2, & q \equiv 1 \mod 4, \\
 0, & q \equiv 2 \mod 4, \\
 \Z/6, & q \equiv 3 \mod 4, \\
 0, & q \equiv 0 \mod 4, \medspace q > 0;
 \end{cases} 
 & \quad &
  \Homol_q({\mathcal{S}}_3; \thinspace \Z/3) &\cong& 
\begin{cases}
 \Z/3, & q = 0, \\
 0, & q \equiv 1 \mod 4, \\
 0, & q \equiv 2 \mod 4, \\
 \Z/3, & q \equiv 3 \mod 4, \\
 \Z/3, & q \equiv 0 \mod 4, \medspace q > 0;
 \end{cases} 
 &\quad & \Homol_q({\mathcal{S}}_3; \thinspace \Z/2) &\cong& \Z/2, q \in \N \cup \{0\};\\
&  \Homol_q({\mathcal{A}}_4; \thinspace \Z) &\cong& 
\begin{cases}
 \Z, & q = 0, \\
 (\Z/2)^k \oplus \Z/3, & q = 6k+1, \\
 (\Z/2)^k \oplus \Z/2, & q = 6k+2, \\
 (\Z/2)^k \oplus \Z/6, & q = 6k+3, \\
 (\Z/2)^k , & q = 6k+4, \\
 (\Z/2)^k \oplus \Z/2 \oplus \Z/6, & q = 6k+5, \\
 (\Z/2)^{k+1}, & q = 6k+6. \\
 \end{cases} 
& \qquad &
\Homol_q({\mathcal{A}}_4; \thinspace \Z/2) &\cong&
 \begin{cases}
\Z/2, & q = 0, \\
(\Z/2)^{2k}, & q = 6k+1, \\
(\Z/2)^{2k+1}, & q = 6k+2, \\
(\Z/2)^{2k+2}, & q = 6k+3, \\
(\Z/2)^{2k+1} , & q = 6k+4, \\
(\Z/2)^{2k+2}, & q = 6k+5, \\
(\Z/2)^{2k+3}, & q = 6k+6. \\
\end{cases} 
 &\quad & \Homol_q({\mathcal{A}}_4; \thinspace \Z/3) &\cong &\Z/3, q \in \N \cup \{0\}.
\end{alignat*}
\normalsize
\end{Lem}
Using the Universal Coefficient Theorem, we see that in degrees $q \geq 1$, 
the homology with trivial $\Z/4$--coefficients is isomorphic to the homology with trivial $\Z/2$--coefficients for the groups listed above.
The stabilizers of the points inside $\Hy$ are finite and hence of the above listed types. 
The stabilizers of the singular points are isomorphic to $\Z^2$, which has homology $ \Homol_q(\Z^2; \thinspace \Z) \cong
$\scriptsize$ \begin{cases}
0, & q \ge 3, \\
\Z, & q = 2, \\
\Z^2, &  q= 1.
\end{cases} $ \normalsize
\\
The maps induced on homology by inclusions of the stabilizers determine the $d^1$-differentials of the spectral sequence from section \ref{SpecSeq}.
\begin{obs}\label{obs}
The three images in $\Homol_2( {\mathcal{D}}_2; \thinspace  \Z/2)$ of the non-trivial element of $\Homol_2(\Z/2; \thinspace  \Z/2)$ under the maps induced by the inclusions of the three order-2-subgroups of ${\mathcal{D}}_2$ are linearly independent, but the three images of the non-trivial element of $\Homol_2( \Z/2; \thinspace  \Z/4)$ are linearly dependent in $\Homol_2( {\mathcal{D}}_2; \thinspace  \Z/4)$.\\
More precisely, there is a canonical basis for $\Homol_2( {\mathcal{D}}_2; \thinspace  \Z/2) \cong (\Z/2)^3$ coming from the resolution for $\mathcal{D}_2$ used by \cite{SchwermerVogtmann}, associated to the decomposition $\mathcal{D}_2\cong\Z/2\times\Z/2$. One checks by direct calculation that in this basis, the inclusions of the three order-2-subgroups in ${\mathcal{D}}_2$ induce the images
\begin{center} \scriptsize 
$\left\{\mathbf{0}, \left( \begin{array}{c} 1 \\ 0 \\ 0 \end{array} \right) \right\}$,
$\left\{\mathbf{0}, \left( \begin{array}{c} 1 \\ 1 \\ 1 \end{array} \right) \right\}$, \normalsize and \scriptsize
$\left\{\mathbf{0}, \left( \begin{array}{c} 0 \\ 0 \\ 1 \end{array} \right) \right\}$
\normalsize in $\Homol_2( {\mathcal{D}}_2; \thinspace  \Z/2)$; \end{center}
and in the basis coming from the same resolution used for $\Z/4$--coefficients these images are 
\begin{center} \scriptsize 
$\left\{\mathbf{0}, \left( \begin{array}{c} 1 \\ 0 \\ 0 \end{array} \right) \right\}$,
$\left\{\mathbf{0}, \left( \begin{array}{c} 1 \\ 0 \\ 1 \end{array} \right) \right\}$, \normalsize and \scriptsize
$\left\{\mathbf{0}, \left( \begin{array}{c} 0 \\ 0 \\ 1 \end{array} \right) \right\}$
\normalsize in $\Homol_2( {\mathcal{D}}_2; \thinspace  \Z/4)$. \end{center}

The difference between the cases $\Z/2$ and $\Z/4$ comes from the behavior of the kernels of the differential maps.
\end{obs}
\begin{Lem}[Schwermer/Vogtmann \cite{SchwermerVogtmann}] \label{inducedMaps} %{$ $}\\
Let $C\in \{\Z\}\cup \{\Z/n: n = 2,3,4\}$. Consider group homology with trivial $C$-coefficients. Then the following hold. 
\begin{enumerate}
\item Any inclusion $\Z/2 \to {\mathcal{S}}_3$ induces an injection on homology.
\item An inclusion $\Z/3 \to {\mathcal{S}}_3$ induces an injection on homology in degrees congruent to $3$ or $0 \mod 4$, and is zero otherwise.
\item Any inclusion $\Z/2 \to {\mathcal{D}}_2$ induces an injection on homology in all degrees.
\item An inclusion $\Z/3 \to {\mathcal{A}}_4$ induces injections on homology in all degrees.
\item In the case $C\in \{\Z\}\cup \{\Z/n: n =2,3\}$, an inclusion $\Z/2 \to {\mathcal{A}}_4$ induces injections on homology in degrees greater than $1$, and is zero on $\Homol_1$.\\
In the case $C=\Z/4$, the same holds in homology degrees $q\ne 2$.\\
An inclusion $\Z/2 \to {\mathcal{A}}_4$ induces the zero map on $\Homol_2(-; \thinspace  \Z/4)$.
\end{enumerate}
\end{Lem}
\begin{proof}[Sketch of proof]
Schwermer and Vogtmann prove this for $C = \Z$, and leave it to the reader in the case $C = \Z/2$.  Details for the latter case can be found in~\cite{RahmPhDthesis}. In the following, we are going to give the main arguments.
\begin{enumerate}
\item This follows for all coefficients from the fact that the group extension \\
$1\to\Z/3\to\mathcal{S}_3\to \Z/2\to 1$ has the property that any inclusion, $\Z/2\to\mathcal{S}_3$, composed with its quotient map is the identity on $\Z/2$.
 \item The assertion is trivial for $C\in\{\Z/2, \Z/4\}$ because then $\Homol_q(\Z/3; \thinspace C) = 0$ for $q\geq 1$ by the Universal Coefficient Theorem. For $C=\Z/3$, one computes the Lyndon/Hochschild/Serre spectral sequence with $\Z/3$-coefficients associated to the extension \mbox{$1\to\Z/3\to\mathcal{S}_3\to \Z/2\to 1$.} Its $E^2$-term $E^2_{p,q}=\Homol_p(\Z/2; \thinspace  \Homol_q(\Z/3; \thinspace  \Z/3))$ is concentrated in the column $p=0$; special care has to be taken with the action of $\Z/2$ on $\Homol_q(\Z/3; \thinspace  \Z/3)$. So $E^2_{p,q}\cong E^\infty_{p,q}$, and the assertion follows from a computation of the map \mbox{$\Homol_q(\Z/3; \thinspace \Z/3)\to\Homol_0(\Z/2; \thinspace \Homol_q(\Z/3; \thinspace \Z/3))$,} i.~e. the projection onto the coinvariants.
\item Similar to (1), this is an immediate consequence of the fact that $1\to\Z/2\to\mathcal{D}_2\to\Z/2\to 1$ splits.
\item Similar to (1) and (3), this follows for all coefficients from the fact that any inclusion $\Z/3\to\mathcal{A}_4$ composed with the unique quotient map $\mathcal{A}_4\to\Z/3$ is an isomorphism, hence induces an isomorphism in homology.
\item The assertion is trivial for $C = \Z/3$ because then $\Homol_q(\Z/2; \thinspace C) = 0$ for $q\geq 1$. For \mbox{$C\in \{\Z/2, \Z/4\}$,} one considers the factorization of the inclusion $\Z/2\to\mathcal{A}_4$ as $\Z/2\to\mathcal{D}_2\to\mathcal{A}_4$ where the first map is one out of three possible inclusions $\Z/2\to\mathcal{D}_2$, denoted by $\alpha, \beta, \gamma$. By (3), $\alpha, \beta$ and $\gamma$ induce injections. Furthermore, one considers the spectral sequence with $C$-coefficients associated to the extension $1\to \mathcal{D}_2\to \mathcal{A}_4\to \Z/3\to 1$. Similar to the case considered in (2), the $E^2$-term $E^2_{p,q}=\Homol_p(\Z/3; \thinspace  \Homol_q(\mathcal{D}_2; \thinspace  C))$ is concentrated in the column $p=0$, thus $E^2_{p,q}\cong E^\infty_{p,q}$, and the map $\Homol_q(\Z/2; \thinspace C)\to\Homol_q(\mathcal{A}_4; \thinspace C)$ is written as the composition
 $$
 \Homol_q(\Z/2; \thinspace  C)\to\Homol_q(\mathcal{D}_2; \thinspace C)\to\Homol_0(\Z/3; \thinspace  \Homol_q(\mathcal{D}_2; \thinspace  C))\cong\Homol_q(\mathcal{A}_4; \thinspace C)
$$
where the first map is $\alpha_*, \beta_*$ or $\gamma_*$ and the second one is the projection onto the $\Z/3$-coinvariants. From this, the statement can be directly deduced for $q\ne 2$. For the case $q=2$, denote the generator of $\Homol_2(\Z/2; \thinspace  C)$ by $x$. The action of $\Z/3$ on $\mathcal{D}_2$ comes from conjugation within $\mathcal{A}_4$ and permutes the three non-trivial elements. There is an automorphism $\phi$ of $\mathcal{D}_2$ given by the action of a generator of $\Z/3$, such that $\phi\circ\alpha=\beta$. Then $\phi_*(\alpha_*(x))=(\phi\circ\alpha)_*(x)=\beta_*(x)$ and $\phi_*(\beta_*(x))=\gamma_*(x)$. For $C=\Z/4$, observation \ref{obs} implies that $\alpha_*(x)=\gamma_*(x)-\beta_*(x)=\gamma_*(x)-(\phi^2)_*(\gamma_*(x))$, and thus $\alpha_*(x)$ is in $\text{Im}(1-\phi_*)=\text{Im}(1-(\phi^2)_*)$, hence is zero in the coinvariants. Therefore, the same holds for $\beta_*(x)$ and $\gamma_*(x)$, and the assertion follows. For $C=\Z/2$, one computes with the help of observation \ref{obs} that $\alpha_*(x)\not\in\text{Im}(1 -\phi_*)$; thus, the same holds for $\beta_*(x)$ and $\gamma_*(x)$ and the assertion follows.
% In order to compute the latter projection, one checks by means of the following simple argument that the $\Z/3$-action permutes the three possible images (see (3) above) of the non-trivial element $x$ of $\Homol_2(\Z/2; \Z/2)$ by means of the following simple argument. Denote the three inclusions $\Z/2\to\mathcal{D}_2$ by $\alpha, \beta, \gamma$. Then $\alpha_*(\beta_*(x))=(\alpha\circ\beta)_*(x)$ and so on. The canonical basis is given by the resolution of $\mathcal{D}_2$ in \cite{SchwermerVogtmann}, and this resolution is used to compute $\alpha_*(x)$. It turns out that $\alpha_*(x)$ shows different behavior according to $n$ (see observation \ref{obs}).
\end{enumerate}
\end{proof}

\subsection{The mass formula for the equivariant Euler characteristic}$\text{ }$\\
We will use the Euler characteristic to check the geometry of the quotient $\Gamma \backslash X$.
Recall the following definitions and proposition, which we include for the reader's convenience.
\begin{df}[Euler characteristic]
Suppose $\Gamma'$ is a torsion-free group. Then we define its Euler characteristic as 
$$\chi(\Gamma')=\sum_i (-1)^i \dim \Homol_i(\Gamma'; \thinspace \rationals).$$
Suppose further that $\Gamma'$ is a torsion-free subgroup of finite index in a group $\Gamma$.
Then we define the {\em Euler characteristic} of $\Gamma$  as 
$$\chi(\Gamma)= \frac{\chi(\Gamma')}{[\Gamma : \Gamma']}.$$
This is well-defined because of \cite{Brown}*{IX.6.3}.
\end{df}
\begin{df}[Equivariant Euler characteristic]
Suppose $X$ is a $\Gamma$-complex such that
\begin{enumerate}
\item
every isotropy group $\Gamma_\sigma$ is of finite homological type;
\item
$X$ has only finitely many cells mod $\Gamma$.
\end{enumerate}
Then we define the $\Gamma$-{\em equivariant Euler characteristic} of $X$ as
$$ \chi_\Gamma(X) := \sum\limits_\sigma (-1)^{\mathrm{dim}\sigma} \chi(\Gamma_\sigma),$$
where $\sigma$ runs over the orbit representatives of cells of $X$.
\end{df}
\begin{prop}[\cite{Brown}*{IX.7.3 e'}]
Suppose $X$ is a $\Gamma$-complex such that $ \chi_\Gamma(X)$ is defined.
If $\Gamma$ is virtually torsion-free, then $\Gamma$ is of finite homological type and $ \chi(\Gamma) =  \chi_\Gamma(X).$
\end{prop}
Let now $\Gamma$ be $\text{PSL}_2\bigl(\mathcal{O}_{\rationals\left[\sqrt{-m}\thinspace\right]}\bigr)$. Then the above proposition applies to $X$ taken to be Fl\"oge's (or still, Mendoza's) $\Gamma$-equivariant deformation retract of $\Hy$, because $\Gamma$ is virtually torsion-free by Selberg's lemma. Using $\chi(\Gamma_\sigma) = \frac{1}{\mathrm{card}(\Gamma_\sigma)}$ for $\Gamma_\sigma$ finite, the fact that the singular points have stabilizer $\Z^2$, and the torsion-free Euler characteristic 
$$\chi(\Z^2) = \sum\limits_i (-1)^i \mathrm{rank}_\Z (\Homol_i \Z^2) = 1-2+1 = 0,$$ we get the formula
$$ \chi(\Gamma) = 
\sum\limits_\sigma (-1)^{\mathrm{dim}\sigma} \frac{1}{\mathrm{card}(\Gamma_\sigma)},$$
where $\sigma$ runs over the orbit representatives of cells of $X$ with finite stabilizers.
\begin{prop} \label{Euler_characteristic_vanishes}
The Euler characteristic $\chi(\Gamma)$ vanishes. 
\end{prop}
\begin{rem} \label{vanish}
Thus, the formula
$$
0  = \sum\limits_\sigma (-1)^{\mathrm{dim}\sigma} \frac{1}{\mathrm{card}(\Gamma_\sigma)},
$$
allows to check the joint data of the geometry of the fundamental domain, cell stabilizers and cell identifications.
\end{rem}
\begin{proof}[Proof of proposition \ref{Euler_characteristic_vanishes}]
Denote by $\zeta_K$ the Dedekind $\zeta$-function associated to the number field \\
\mbox{ $K := \rationals\left[\sqrt{-m}\thinspace\right]$.} 
Brown \cite{Brown}*{below (IX.8.7)} deduces the following from Harder's result \cite{Harder}*{p. 453}: 
$$ \chi(SL_n(\mathcal{O}_K)) = \prod\limits_{j=2}^n \zeta_K (1-j),$$
so especially we have $ \chi(SL_2(\mathcal{O}_K)) = \zeta_K (-1).$ 
As $\Gamma$ is a quotient of $SL_2(\mathcal{O}_K)$ by a group of order two, it follows \cite{Bass} that
 $$ \chi(\Gamma) = 2\cdot\chi(SL_2(\mathcal{O}_K)) = 2\cdot\zeta_K (-1).$$
Using the functional equation of $ \zeta_K $ \cite{Narkiewicz} and the fact that $K$ has no real embeddings because it is imaginary quadratic, 
we get  $ \zeta_K (-1) = 0.$
\end{proof}

\begin{rem}
One can prove the above proposition without using the Dedekind zeta function.
This alternative proof applies to \emph{any} cofinite arithmetically defined subgroup $\Gamma$ of $\PSL(2,\C)$. Let $\Gamma'$ denote a torsion-free subgroup of $\Gamma$ of finite index. It is the main theorem of Harder's article on the Gauss-Bonnet theorem \cite{Harder} that the Euler characteristic of $\Gamma'$ is its covolume with respect to the Euler-Poincar\'e form $\mu$ on $\Hy$, i.~e. $\chi(\Gamma')=\int_Y d\mu$, where $Y$ is a fundamental domain for the action of $\Gamma'$ on $\Hy$. This extends the classical Gauss-Bonnet theorem from the theory of the Euler-Poincar\'e form, see \cite{Serre}*{paragraph 3} (where the theorem is hidden as the existence assertion of the Euler-Poincar\'e measure) to non-cocompact but cofinite discrete subgroups. The measure $\mu$ is a fundamental datum associated to the symmetric space, without reference to any discrete group. In \cite{Serre}*{paragraph 3,2a} it is shown that $\mu=0$ on any odd-dimensional space. Since $\text{dim }\Hy=3$, 
 we have $\chi(\Gamma')= \chi(\Gamma) = 0$.
\end{rem}

\section{Computations of the integral homology of $\text{PSL}_2\bigl(\mathcal{O}_{\rationals\left[\sqrt{-m}\thinspace\right]}\bigr)$}
Throughout this section, the action on the homology coefficients is trivial because the stabilizers fix the cells pointwise. We mean $\Z$-coefficients wherever we do not mention the coefficients. 
Throughout, we label the singular point in the fundamental domain by $s$ and use the notation
$$ \otimes_\sigma  :=  \otimes_{\Z[\Gamma_\sigma]}.$$ 
We write $\mathcal{D}_2$ for the Klein four group, $\mathcal{S}_3$ for the permutation group on three objects and $\mathcal{A}_4$ for the alternating group on four objects. 
\\
We have \mbox{$\Gamma = \text{PSL}_2(\mathcal{O}_{\rationals\left[\sqrt{-m}\thinspace\right]})
= \text{PSL}_2(\Z[\omega])$}
 with $\omega := \sqrt{-m}$ in the cases $m=5, 6, 10,13$. The coordinates in hyperbolic space of the vertices of the fundamental domains have been computed by Bianchi \cite{Bianchi}. 
There, they are listed up to complex conjugation for $m=5,6,15$; and for $m=10,13$, the reader has to divide out the reflection called \emph{riflessione impropria} by Bianchi.

\subsection{$m=13$}
\begin{wrapfigure}{r}{25mm}
  \centering 
  \includegraphics[width=22mm]{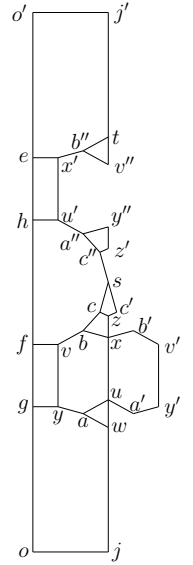} 
  \caption{The fundamental domain for $m=13$} \label{fundamental domain 13}
\end{wrapfigure}
We make the following definitions. \\
\scriptsize
$\begin{array}{lllllllll}
\\
A&:=& \pm
\begin{pmatrix} 9
& 
7 \omega
\\ 
\omega
& 
-10
\end{pmatrix},
&B&:= &\pm \begin{pmatrix} -2
-\omega
& 
2
-\omega
\\ 
4
 & 
2
+1 \omega\end{pmatrix},
&C&:=&\pm \begin{pmatrix} -1
- \omega
 & 
8
- \omega
 \\ 
3
 & 
1
+2 \omega\end{pmatrix},
\\ \\
D&:=& \pm \begin{pmatrix} 5
 & 
 2 \omega
 \\ 
  \omega
 & 
-5\end{pmatrix},
&E&:=& \pm \begin{pmatrix} - \omega
 & 
6
 \\ 
2
 & 
 \omega\end{pmatrix},
&J&:= &\pm \begin{pmatrix}  & 
1
 \\ 
-1
 & \end{pmatrix},
\\ \\
S&:= &\pm \begin{pmatrix}  & -1 \\ 1 & 1\end{pmatrix},
&K&:= &\pm \begin{pmatrix} 11+4\omega
 &
-17+7\omega 
 \\
-8+\omega 
 & 
-10-3\omega\end{pmatrix},
&M&:= &\pm \begin{pmatrix} 4
-2 \omega
 & 
12
+ \omega
 \\ 
4
+ \omega
 & 
-4
+2 \omega\end{pmatrix},
\\ \\
U&:= &\pm \begin{pmatrix}  1 & \omega \\ & 1 \end{pmatrix},
&V&:=&
\pm \begin{pmatrix}
- \omega
 & 
6
- \omega
 \\ 
2
 & 
2
+ \omega
\end{pmatrix},
&W&:=&
\pm \begin{pmatrix}
14
- \omega
 & 
13
+6 \omega
 \\ 
2 \omega
 & 
-12
+ \omega
\end{pmatrix},
\\ \\
P &:=&\thinspace V^{-1}D, 
&T&:=&\thinspace P^{-1}S^2,
&R&:=&\thinspace TU^{-1}S^2U. \\ \\
\end{array}$ \normalsize  \\
\label{identifications13}
Vertices with the same letter in the fundamental domain displayed in figure \ref{fundamental domain 13} are identified by the action of $\Gamma$, for instance, $y$ is identified with $y'$ and $y''$ and so on. This yields relations between the matrices in the same way as shown by \cite{Floege}. Amongst these relations, we will use \mbox{$ T = CKCA(CKC)^{-1}$,} $V^{-1} = CAC^{-1}M$ and $S^2 = BS^{-1}BS $ in our further calculations, in particular in the computation of the $d^2$-differential. Note that the 2-cell $(y'',a'',u',x',b'',v'')$ is identified with the 2-cell $(y,a,u,x,b,v)$, hence only one of them can be in the fundamental domain. 
The matrix $U$ acts as a vertical translation by $-\omega$. Furthermore, we will use the identifications $ C \cdot x' = x$, $ U \cdot j' = j$, $ C \cdot y' = y $ and
\mbox{$ K \cdot z = z' $.} \\
Amongst the edge identifications, we will use \mbox{$CAC^{-1} \cdot (c,z) = (c',z)$,} \quad \mbox{$V^{-1} \cdot (s,c)=(s,c'),$} \quad
\mbox{$ CAC^{-1} \cdot (b,x) = (b',x) $,} \quad \mbox{$V^{-1}\cdot (b,v) = (b',v') $,} \quad
\mbox{$P\cdot (y,v) = (y',v') $,} \quad \mbox{$S^{2}\cdot(a,y) =(a',y')$, and}
\mbox{$B\cdot (a,u) = (a',u) $. 

There are seventeen orbits of vertices, which have the following stabilizers.}
% Stabilizers of the vertex orbit representatives.
$$ \scriptsize
\begin{array}{llllll}
& \Gamma_{o} &=&  &              & \langle J | J^2 = 1 \rangle \cong {\mathbb{Z} /2},\\ 
& \Gamma_{a} &=&  &              & \langle S^{-1}BS | (S^{-1}BS)^2 = 1 \rangle \cong {\mathbb{Z} /2},\\ 
& \Gamma_{b} &=&  \Gamma_{c} &=& \langle M | M^2 = 1 \rangle \cong {\mathbb{Z} /2} ,\\
& \Gamma_{u} &=&  &              & \langle B | B^2 = 1 \rangle \cong {\mathbb{Z} /2},\\ 
& \Gamma_{v} &=&  &              & \langle D | D^2 = 1 \rangle \cong {\mathbb{Z} /2},\\ 
& \Gamma_{f} &=&  &              & \langle D, E | D^2 = E^2 = (DE)^2 = 1 \rangle \cong {{\mathcal{D}}_2},\\ 
& \Gamma_{h} &=&  &              & \langle E, AU^{-1}JU | E^2 = (AU^{-1}JU)^2 = (EAU^{-1}JU)^2 = 1 \rangle \cong {{\mathcal{D}}_2},\\ 
& \Gamma_{e} &=&  &              & \langle A, U^{-1}JU | A^3= (U^{-1}JU)^2 = (AU^{-1}JU)^2 = 1 \rangle \cong {{\mathcal{S}}_3},\\ 
& \Gamma_{g} &=&  &              & \langle J, T | J^2 = T^3 = (JT)^2 = 1 \rangle \cong {{\mathcal{S}}_3},\\ 
& \Gamma_{t} &=&  &              & \langle R, U^{-1}SU | R^2 = (U^{-1}SU)^3 = (RU^{-1}SU)^2 = 1 \rangle \cong {{\mathcal{S}}_3},\\ 
& \Gamma_{w} &=&  &              & \langle B, S | B^2 = S^3 = (BS)^2 = 1 \rangle \cong {{\mathcal{S}}_3},\\ 
& \Gamma_{j} &=&  &              & \langle S | S^3 = 1 \rangle \cong {\mathbb{Z} /3},\\ 
& \Gamma_{x} &=&  \Gamma_{z} &=& \langle CAC^{-1} | (CAC^{-1})^3 = 1 \rangle  \cong {\mathbb{Z} /3},\\ 
& \Gamma_{y} &=&  &              & \langle T | T^3 = 1 \rangle \cong {\mathbb{Z} /3},\\
& \Gamma_{s} &=&  &              & \langle  V,W | VW = WV \rangle  \cong {\mathbb{Z}^2}. 
\end{array} $$ \normalsize
There are twenty-eight orbits of edges. \\
The edge stabilizers isomorphic to $\Z/3$ are given on the chosen representatives as
% 3-primary part of representative edge stabilizers for m = 13
\scriptsize $$
\begin{array}{llll}
&\Gamma_{(e,x')} &=& 
\langle A |A^3 = 1 \rangle
 \cong {\mathbb{Z} /3}, \\
&\Gamma_{(x,z)} &=& 
\langle CAC^{-1} | (CAC^{-1})^3 = 1 \rangle 
 \cong {\mathbb{Z} /3},\\
&\Gamma_{(g,y)} &=& 
\langle T | T^3 = 1 \rangle
 \cong {\mathbb{Z} /3},\\
&\Gamma_{(j,w)} &=& \langle S| S^3 = 1 \rangle
 \cong {\mathbb{Z} /3},\\
 &\Gamma_{(t,j')} &=& 
\langle U^{-1}SU| (U^{-1}SU)^3 = 1 \rangle
 \cong {\mathbb{Z} /3}, \\
&\Gamma_{(y',z')} &=& 
\langle KCA(KC)^{-1} | (KCA(KC)^{-1})^3 = 1 \rangle
 \cong {\mathbb{Z} /3},
 \end{array} 
$$ \normalsize
and the edge stabilizers isomorphic to $\Z/2$ are given on the chosen representatives as
 \scriptsize $$
\begin{array}{llll}
&\Gamma_{(f,v)} &=& 
\langle D | D^2 = 1 \rangle
 \cong {\mathbb{Z} /2},\\
&\Gamma_{(h,u')} &=& 
\langle EAU^{-1}JU | (EAU^{-1}JU)^2 = 1 \rangle
 \cong {\mathbb{Z} /2},\\
&\Gamma_{(t,b'')} &=& 
\langle R | R^2 = 1 \rangle
 \cong {\mathbb{Z} /2},\\
&\Gamma_{(w,a)} &=&
\langle S^{-1}BS | (S^{-1}BS)^2 = 1 \rangle
 \cong {\mathbb{Z} /2},\\
&\Gamma_{(b,c)} &=& 
\langle M | M^2 = 1 \rangle
 \cong {\mathbb{Z} /2},\\
&\Gamma_{(a'',c'')} &=& 
\langle C^{-1}S^{-1}BSC | (C^{-1}S^{-1}BSC)^2 = 1 \rangle 
 \cong {\mathbb{Z} /2},\\
&\Gamma_{(v',t)} &=& 
\langle RU^{-1}SU | (RU^{-1}SU)^2 = 1 \rangle
 \cong {\mathbb{Z} /2},\\
&\Gamma_{(w,u)} &=& 
\langle B | B^2 = 1 \rangle
 \cong {\mathbb{Z} /2},\\
&\Gamma_{(h,e)} &=&
\langle AU^{-1}JU | (AU^{-1}JU)^2 = 1 \rangle
 \cong {\mathbb{Z} /2},\\
&\Gamma_{(g,f)} &=& 
\langle DE | (DE)^2 = 1 \rangle
 \cong {\mathbb{Z} /2},\\
&\Gamma_{(f,h)} &=& 
\langle E | E^2 = 1 \rangle
 \cong {\mathbb{Z} /2},\\
&\Gamma_{(o,g)} &=& 
\langle J | J^2 = 1 \rangle 
 \cong {\mathbb{Z} /2},\\
&\Gamma_{(o',e)} &=& 
\langle U^{-1}JU | (U^{-1}JU)^2 = 1 \rangle
 \cong {\mathbb{Z} /2}.
\end{array} 
$$ \normalsize
We find nine edge orbits with the trivial stabilizer, thirteen edge orbit representatives with stabilizer type $\Z/2$, and six with stabilizer type $\Z/3$.
The singular vertex has stabilizer type $\Z^2$, and
there are six vertex orbit representatives with stabilizer type $\Z/2$, two with $\mathcal{D}_2$, four with $\mathcal{S}_3$ and four with $\Z/3$.
\mbox{Furthermore, there are twelve orbits of faces with trivial stabilizers.}
\mbox{The above data gives the $\Gamma$-equivariant Euler characteristic of $X$, in accordance with remark \ref{vanish}:} 
$$ \chi_\Gamma(X) = \frac{6}{2} + \frac{4}{3} +\frac{2}{4} +\frac{4}{6}
-9 -\frac{13}{2} -\frac{6}{3}
+ 12 = 0. $$ 
\subsubsection{The bottom row of the $E^1$-term}$\text{ }$\\
\begin{wrapfigure}{r}{72mm}
\caption{$ (d^1_{1,q})_{(2)}$}  \label{13-by-13} \tiny
$\left(  \begin{array}{*{16}{c}}
1 &  &  &  &  &  &  &  &  & 1 &  &  &  \\ 
 &  &  &  &  &  &  &  &  & \vdots &  &  &  \\ 
 &  &  &  &  &  &  &  &  & 1 & 1 &  &  \\ 
 & 1 &  &  &  &  &  &  & 1 &  &  &  &  \\ 
 & \vdots &  &  &  &  &  &  &  &  &  &  &  \\ 
 & 1 &  &  &  &  &  &  &  &  & 1 &  &  \\ 
 &  &  &  &  &  &  &  &  &  &  & 1 & 1 \\ 
 &  &  &  &  &  &  &  & 1 &  &  &  & 1 \\ 
 &  &  &  &  &  &  &  &  & 1 &  & 1 &  \\ 
 &  & 1 &  &  &  & 1 &  &  &  &  &  &  \\ 
 & 1 &  &  &  &  &  & 1 &  &  &  &  &  \\ 
1 &  &  &  &  &  & 1 &  &  &  &  &  &  \\ 
 &  &  & 1 &  &  &  & 1 &  &  &  &  &  \\ 
 &  &  & 1 &  & 1 &  &  &  &  &  &  &  \\ 
 &  & 1 &  & 1 &  &  &  &  &  &  &  &  \\ 
 &  &  &  & 1 & 1 &  &  &  &  &  &  & 
\end{array} \right)  $
\normalsize
\end{wrapfigure}
We obtain for the row $q = 0$ in the columns $p = 0, 1, 2$:

$
 \Z^{17} \xleftarrow{\ d^1_{1,0}\ } \Z^{28}\xleftarrow{\ d^1_{2,0}\ } \Z^{12}, 
$
\\
where the only occurring elementary divisor is 1, with multiplicity sixteen
for $d^1_{1,0}$, and with multiplicity ten \mbox{for $d^1_{2,0}$.}

\subsubsection{The odd rows of the $E^1$-term}$\text{ }$\\

We set the goal to determine the morphism \\
$
\bigoplus_{\sigma \in \Gamma\backslash X^0} \Homol_q (\Gamma_\sigma)
\xleftarrow{\ d^1_{1,q}\ }
\bigoplus_{\sigma \in \Gamma\backslash X^1} \Homol_q(\Gamma_\sigma)
$ for \mbox{odd $q\ge 1$.} By lemma \ref{finiteSubgroups}, the torsion only occurs at the primes 2 and~3. For each $m$, we are going to treat these primes separately. For each of them, lemma \ref{inducedMaps} specifies the effect on homology of the vertex inclusion, \\
\mbox{thus allowing us
to determine the matrix of $d^1_{1,q}$ in the basis determined by the edge respectively} \\
 \mbox{vertex
stabilizers  in a row-by-row fashion.
For $m=13$ and odd $q$, the map $d^1_{1,q}$ is on the 2-primary} \\
\mbox{part a homomorphism ${\mathbb{(Z} /2)^{q+13}}
\longleftarrow
(\Z/2)^{13}
$ given by the $q+$13-by-13 matrix in figure \ref{13-by-13},
where we} \\
\mbox{replace the dotted entries `` \scriptsize$\vdots$ \normalsize ''  by $\frac{q-1}{2}$ lines with a ``1''
in the column of the dots, and zeroes in the} \\
\mbox{rest of these lines. Therefore, we have to distinguish the case $q=1$, where $d^1_{1,q}$ has rank 12,} \\
 and the case $q \ge 3$, where it has \mbox{rank 13.}
\\
On the 3-primary part, $d^1_{1,q}$ is a homomorphism
$ \begin{cases}
{\mathbb{(Z} /3)^{4}}
\longleftarrow
(\Z/3)^{6}
&
\mathrm{for} \thinspace q \equiv 1 \mod 4,
\\
{\mathbb{(Z} /3)^{8}}
\longleftarrow
(\Z/3)^{6} 
&
\mathrm{for} \thinspace q \equiv 3 \mod 4.
\end{cases} $ 

\newpage
\begin{wrapfigure}{r}{72mm}
\caption{$ (d^1_{1,q})_{(3)}$}  \label{3-primary part} \scriptsize
$ \begin{array}{c|cccccc}
& (e,x') & (g,y) & (x,z) & (y',z') & (j,w) & (t,j')
\\
\hline
e & -\alpha & 0 & 0 & 0 & 0 & 0
\\
x &  1 & 0 & -1 & 0 & 0 & 0
\\
g & 0 &  -\alpha & 0 & 0 & 0 & 0
\\
y & 0 & 1 & 0 & - 1 & 0 & 0
\\
z & 0 & 0 & 1 &  1 & 0 & 0
\\
j & 0 & 0 & 0 & 0 & -1 &  1
\\
w & 0 & 0 & 0 & 0 & \alpha & 0
\\
t & 0 & 0 & 0 & 0 & 0 & -\alpha,
\end{array}$ \normalsize
\end{wrapfigure}

It is given by the matrix displayed in figure \ref{3-primary part},
where $\alpha = 1$ for $q \equiv 3 \mod 4$ and $\alpha = 0$ else. This matrix has full rank 6 (injectivity) for $q \equiv 3 \mod 4$, and rank 4 (surjectivity) for $q \equiv 1 \mod 4$.
For $q = 1$, there is an additional module $\Homol_1(\Gamma_{s}) \cong {\mathbb{Z}^2}$ on the target side, which can not be hit because the edge stabilizers are only torsion.
\begin{rem}
So, the 3-torsion in $\Homol_1(\Gamma)$ has already been killed by the $d^1$ differential. This is useful for showing that the map
\\
$ \Homol_1(\mathrm{PSL}_2(\Z)) \rightarrow \Homol_1(\Gamma)$
\\
is not injective. In fact, the matrix $S$ of order 3 defines a non-zero element in the abelianization of $\mathrm{PSL}_2(\Z)$ but becomes subject to the relation $S^2 = BS^{-1}BS$ in $\Gamma$ where $B$ is the matrix of order two defined above. Thus, the class of $S$ is zero in $\Gamma^{\text{ab}}$.
\end{rem}
\subsubsection{The even rows of the $E^1$-term}$\text{ }$\\
There is a zero map arriving at
$\bigoplus\limits_{\sigma \in \Gamma\backslash X^0} \Homol_q (\Gamma_\sigma) \cong ({\mathbb{Z} /2})^q
$
for $q$ bigger than 2, and respectively at 
$$\bigoplus\limits_{\sigma \in \Gamma\backslash X^0} \Homol_2 (\Gamma_\sigma) \cong {\mathbb{Z}} \oplus ({\mathbb{Z} /2})^2.
$$

\subsubsection{The $E^2$-term}$\text{ }$\\
In the rows with $q \ge 2$, $E^2_{p,q}$ is concentrated in the columns $p=0$ and $p=1$ given as follows:
\scriptsize
$$
\begin{array}{{ll|cccc}}
q = 4k +1, & q \ge 5 &  (\Z/2)^q & (\Z/3)^2 \\
q\ \text{even}, & q \ge 4 & (\Z/2)^q & 0    \\
q = 4k +3, & q \ge 3 & (\Z/3)^2 \oplus (\Z/2)^q & 0 \\
\dots & & \dots & \dots \\
q=2 & & \Z \oplus (\Z/2)^2 & 0
\end{array}
$$
\normalsize
In the rows $q=0$ and $q=1$,  $E^2_{p,q}$ is concentrated in the columns $p=0,1,2$:
\scriptsize
$$ \xymatrix{
q=1 & \Z^2 \oplus (\Z/2)^2 & ({\mathbb{Z} /3})^2 \oplus {\mathbb{Z} /2} & 0 \\
q=0 &{\mathbb{Z}} & {\mathbb{Z}}^2 & \ar[ull]_{d^2} {\mathbb{Z}^2} 
}$$
\normalsize

\subsubsection{The differential $d^2$}$\text{ }$\\
\label{d2differential}
The only nontrivial $d^2$-arrow is determined on the $E^0$-level by the following maps connecting 
$E^0_{2,0}$ with $E^0_{0,1}$:
\scriptsize
$$\xymatrix{
\bigoplus\limits_{\sigma\in\Gamma\backslash X^0} \Theta_1 \otimes_\sigma \Z 
 &
\ar[l]_{1 \otimes \delta } 
\bigoplus\limits_{\sigma\in\Gamma\backslash X^1} \Theta_1 \otimes_\sigma \Z   
\ar[d]^{d_\Theta \otimes 1} 
\\
&
\bigoplus\limits_{\sigma\in\Gamma\backslash X^1} \Theta_0 \otimes_\sigma \Z 
&
\bigoplus\limits_{\sigma\in\Gamma\backslash X^2} \Theta_0 \otimes_\sigma \Z  
\ar[l]_{1 \otimes \delta } 
}$$
\normalsize
where $d_\Theta$ is the differential of the bar resolution $\Theta_\bullet$ for $\Gamma$, 
and $\delta$ is the differential of Fl\"oge's cellular complex.
The  generators of the abelian group $E^2_{2,0} \cong {\mathbb{Z}}^2 $ are represented by the face
$(c,s,c',z)$ and the union of two faces $(b,x,b',v',y',a',u,a,y,v) =: F$, whose quotients by $\Gamma$ are homeomorphic to 2-spheres. \\
Using the identifications stated in \ref{identifications13}, we compute that the above $d^2$-arrow is induced by $$\delta\left((c,s,c',z)\right) = 
(CAC^{-1}-1) \cdot (c,z) +(V^{-1}-1) \cdot (s,c)$$ and
\scriptsize $$
\delta\left((b,x,b',v',y',a',u,a,y,v) \right)
= \thinspace\thinspace (CAC^{-1} -1) \cdot (x, b) 
 + (V^{-1}-1) \cdot (b,v) 
 +(P-1)\cdot(v,y)
+ (S^2 -1) \cdot (y, a)
+(B -1)\cdot(a,u).
$$ \normalsize
The lift $1 \otimes_F 1 $ 
in $E^0_{2,0}$ of the generator of $E^2_{2,0}$ represented by \\
\mbox{$F = (b,x,b',v',y',a',u,a,y,v)$} is mapped as follows:
\scriptsize
$$\xymatrix@C=2em{
{
\begin{array}{c}
 (1,CAC^{-1})  \otimes_b 1  
 -(1,CAC^{-1})  \otimes_x 1  \\
 + (1,V^{-1}) \otimes_v 1 
 - (1,V^{-1}) \otimes_b 1 \\
  +(1,P) \otimes_y 1  
  -(1,P) \otimes_v 1  \\
+ (1,S^2) \otimes_a 1 
- (1,S^2) \otimes_y 1 \\
 +(1,B)\otimes_u 1 
 -(1,B)\otimes_a 1  
\end{array}
}
& & &
\ar[lll]_{1 \otimes \delta } 
{
\begin{array}{c}
 (1,CAC^{-1})  \otimes_{(x, b)}  1  \\
 + (1,V^{-1}) \otimes_{(b,v)}  1 \\
  +(1,P) \otimes_{(v,y)}  1  \\
+ (1,S^2) \otimes_{(y, a)}  1 \\
 +(1,B)\otimes_{(a,u)}  1 
\end{array}
}
\ar[d]^{d_\Theta \otimes 1} 
&
\\
& & &
{
\begin{array}{c}
 (CAC^{-1} -1)  \otimes_{(x, b)}  1  \\
 + (V^{-1}-1) \otimes_{(b,v)}  1 \\
  +(P-1) \otimes_{(v,y)}  1  \\
+ (S^2 -1) \otimes_{(y, a)}  1 \\
 +(B -1)\otimes_{(a,u)}  1 
\end{array}
}
& & &
1 \otimes_F 1  
\ar[lll]_{1 \otimes \delta } 
}
$$
\normalsize
\subsubsection*{The passage to $E^1$}$\text{ }$\\
We attribute the symbols $t_\sigma$ to the part of this sum lying in 
$\Theta_1 \otimes_\sigma \Z $:
$$
\begin{array}{lll}
&t_x &:=  -(1,CAC^{-1})  \otimes_x 1,\\
&t_b &:=  (1,CAC^{-1})  \otimes_b 1   - (1,V^{-1}) \otimes_b 1,\\
&t_v &:=   (1,V^{-1}) \otimes_v 1  -(1,P) \otimes_v 1,\\
&t_y &:= (1,P) \otimes_y 1  - (1,S^2) \otimes_y 1,\\
&t_a &:=  (1,S^2) \otimes_a 1  -(1,B)\otimes_a 1,\\
&t_u &:= (1,B)\otimes_u 1.\\
\end{array}
$$
With the formula in our corollary \ref{Corollary}, we find the classes $\bar{t_\sigma}$ in 
$\Homol_1(\Theta_* \otimes_\sigma \Z)$ as follows: \\
% $$ t_x = -[CAC^{-1}]  \otimes_{{\mathbb{Z}} \Gamma_{x} } 1 $$
%gives the cycle $$ -\overline{CAC^{-1}} \in 
% \langle \overline{CAC^{-1}} | \quad 3\overline{CAC^{-1}} = 0 \rangle 
% \cong  \Homol_1(\Gamma_x; \Z )  . $$
As $ V^{-1}M = CAC^{-1} $ and 
$\Gamma_b = \langle M|  \quad M^2 = 1 \rangle $,
$$
t_b =  
[CAC^{-1}]  \otimes_b 1   - [V^{-1}] \otimes_b 1
=  [V^{-1}M]  \otimes_b 1   - [V^{-1}] \otimes_b 1
$$ 
gives the cycle $$ \overline{VV^{-1}M} -\overline{VV^{-1}}  = \overline{M} \in 
 \langle \overline{M} | \quad 2\overline{M} = 0 \rangle 
\cong  \Homol_1(\Gamma_b; \Z )  . $$
As $ V^{-1} = PD $ and 
$\Gamma_v = \langle D|  \quad D^2 = 1 \rangle $,
$$
t_v = [V^{-1}] \otimes_v 1  -[P] \otimes_v 1
= [PD] \otimes_v 1  -[P] \otimes_v 1
$$ 
gives the cycle $$ \overline{P^{-1}PD} -\overline{P^{-1}P}  = \overline{D} \in 
 \langle \overline{D} | \quad 2\overline{D} = 0 \rangle 
\cong  \Homol_1(\Gamma_v; \thinspace  \Z )  . $$
%As $ PT = S^{2} $ and 
%$\Gamma_y = \langle T|  \quad T^3 = 1 \rangle $,
%$$
%t_y  
%= [P] \otimes_{{\mathbb{Z}} \Gamma_{y} } 1  - [S^2]\otimes_{{\mathbb{Z}} \Gamma_{y} } 1
%= [P] \otimes_{{\mathbb{Z}} \Gamma_{y} } 1  - [PT] \otimes_{{\mathbb{Z}} \Gamma_{y} } 1
%$$ 
%gives the cycle $$ \overline{P^{-1}P} -\overline{P^{-1}PT}  = -\overline{T} \in 
% \langle \overline{T} | \quad 3\overline{T} = 0 \rangle 
% \cong  \Homol_1(\Gamma_y; \Z )  . $$
As $ S^2 = BS^{-1}BS $ and 
$\Gamma_a = \langle S^{-1}BS|  \quad (S^{-1}BS)^2 = 1 \rangle $,
$$
t_a   
= [S^2] \otimes_a 1  -[B]\otimes_a 1
= [BS^{-1}BS] \otimes_a 1  -[B]\otimes_a 1
$$ 
gives the cycle $$ \overline{B^{-1}BS^{-1}BS} -\overline{B^{-1}B}  = \overline{S^{-1}BS} 
\in  \langle \overline{S^{-1}BS} | \quad 2\overline{S^{-1}BS} = 0 \rangle 
\cong  \Homol_1(\Gamma_a; \thinspace  \Z )  . $$
Finally, $t_u = [B]\otimes_u 1$ gives the cycle $$ \overline{B} \in  \langle \overline{B} | \quad 2\overline{B} = 0 \rangle 
\cong  \Homol_1(\Gamma_u; \thinspace  \Z )  . $$
The term $E^2_{0,1}$ has no 3-torsion, so the 3-torsion part $\bar{t_x} +\bar{t_y}$ of the above sum makes no contribution to the image of $d^2$.\\
The 2-torsion part, $ \overline{t_b} + \overline{t_a} + \overline{t_v} + \overline{t_u} $,
 equals the image 
$$ d^1_{1,1}( \overline{t_{(b,c)}} + \overline{t_{(c'',a'')}} 
+ \overline{t_{(v,f)}} + \overline{t_{(f,h)}} + \overline{t_{(h,u')}} ),$$
where $ \overline{t_\sigma} $ stands for the generator of 
$ \Homol_1(\Gamma_\sigma; \thinspace  \Z ) \cong \Z/2. $ 
Thus it is a boundary and is quotiented to zero on the $E^2$-page.
Hence it makes no contribution either to the image of $d^2$, so we obtain that $ d^2(F) = 0. $\\
The lift $1 \otimes_{(c,s,c',z)}  1 $ 
of the generator $(c,s,c',z)$ 
is mapped as follows:
\scriptsize
$$\xymatrix@C=2em{
{
\begin{array}{c}
 (1,CAC^{-1})  \otimes_z 1  \\
 -(1,CAC^{-1})  \otimes_c 1  \\
 + (1,V^{-1}) \otimes_c 1 \\
 - (1,V^{-1}) \otimes_s 1
\end{array}
}
& &
\ar[ll]_{1 \otimes \delta } 
{
\begin{array}{c}
 (1,CAC^{-1})  \otimes_{(c, z)}  1  \\
 + (1,V^{-1}) \otimes_{(s,c)} 1
\end{array}
}
\ar[d]^{d_\Theta \otimes 1} 
&
\\
& &
{
\begin{array}{c}
 (CAC^{-1} -1)  \otimes_{(c, z)} 1  \\
 + (V^{-1}-1) \otimes_{(s,c)} 1  
\end{array}
}
& &
1 \otimes_{(c,s,c',z)}  1  
\ar[ll]_{1 \otimes \delta } 
}
$$
\normalsize
\subsubsection*{The passage to $E^1$}$\text{ }$\\
We attribute the symbols $t_\sigma$ to the part of this sum lying in 
$\Theta_1 \otimes_\sigma \Z  $:
$$
\begin{array}{lll}
&t_z &:=  (1,CAC^{-1})  \otimes_z 1  ,\\
&t_c &:=   (1,V^{-1}) \otimes_c 1
 -(1,CAC^{-1})  \otimes_c 1 ,\\
&t_s &:=    - (1,V^{-1}) \otimes_s 1.
\end{array}
$$
With the formula in our corollary \ref{Corollary}, we find the classes $\bar{t_\sigma}$ in 
$\Homol_1(\Theta_* \otimes_\sigma \Z )$ as follows:
\\
As $ V^{-1}M = CAC^{-1} $ and 
$\Gamma_c = \langle M|  \quad M^2 = 1 \rangle $,
$$
t_c =  
[V^{-1}] \otimes_c 1 
-[CAC^{-1}]  \otimes_c 1 
=  [V^{-1}] \otimes_c 1 
-[V^{-1}M]  \otimes_c 1 
$$ 
gives the cycle $$ \overline{VV^{-1}} -\overline{VV^{-1}M}  = -\overline{M} \in 
 \langle \overline{M} | \quad 2\overline{M} = 0 \rangle 
\cong  \Homol_1(\Gamma_c; \thinspace  \Z )  . $$
Finally,
$$ t_s 
= -[V^{-1}]\otimes_s 1 $$
gives the cycle
$$ \overline{V} \in  \langle \overline{V}, \overline{W}  \rangle 
\cong  \Homol_1(\Gamma_s; \thinspace  \Z )  \cong \Z^2. $$
The term $E^2_{0,1}$ has no 3-torsion, so the 3-torsion part $\overline{t_z}$ of the above sum makes no contribution to the image of $d^2$. \\
However the 2-torsion part, $ \overline{t_c} = \overline{M}$, passes to the $E^2$-page because no chain of edges can have the single point $c$ as its boundary.
Furthermore, $\overline{V}$ is one of the generators of the free part of $E^2_{0,1} \cong \Z^2 \oplus (\Z/2)^2$,  
so we obtain $ d^2\left((c,s,c',z)\right) = \overline{M}+\overline{V}$, which is of infinite order and has the following property: there is no element $\eta \in E^2_{0,1}$ with $k\eta = \overline{M}+\overline{V}$ for an integer $k>1$. 
As we have seen that $ d^2(F) =0,$ we obtain the quotient $$E^3_{0,1} \cong \Z \oplus (\Z/2)^2.$$
Hence we obtain for integral homology the following short exact sequences:
$$\small
\begin{cases}
 0 \rightarrow (\Z/2)^{q} \rightarrow \Homol_q(\Gamma; \thinspace \Z) \rightarrow (\Z/3)^2 \rightarrow 0
,\quad &  \thinspace q = 4k+2, 
\\
 0 \rightarrow (\Z/2)^{q}   \rightarrow \Homol_q(\Gamma; \thinspace \Z) \rightarrow  0
,\quad &  \thinspace q = 4k +1, 
\\
 0 \rightarrow (\Z/2)^{q} \rightarrow \Homol_q(\Gamma; \thinspace \Z) \rightarrow 0 
,\quad &  \thinspace q = 4k+4, 
\\
 0 \rightarrow (\Z/3)^2 \oplus (\Z/2)^{q}    \rightarrow \Homol_q(\Gamma; \thinspace \Z) \rightarrow  0
,\quad &  \thinspace q = 4k +3, 
\\
 0 \rightarrow \Z \oplus (\Z/2)^2 \rightarrow \Homol_2(\Gamma; \thinspace \Z) 
\rightarrow \Z \oplus (\Z/3)^2 \oplus \Z/2 \rightarrow 0,
\\
 0 \rightarrow \Z\oplus (\Z/2)^2 \rightarrow \Homol_1(\Gamma; \thinspace \Z) \rightarrow \Z^2 \rightarrow 0.
\end{cases}
$$\normalsize 
We will resolve the ambiguity in the torsion part of the group extension $ \Homol_2(\Gamma; \thinspace  \Z)$  by a reflection like the one 
on \cite{SchwermerVogtmann}*{page 587}, 
for which we have to recompute the spectral sequence with $\Z/2$--, $\Z/3$-- and $\Z/4$--coefficients.
The free part is unambiguous, as we can see from tensoring with $\rationals$. 
\subsubsection{The $E^1$-term with $\Z/2$-coefficients}$\text{ }$\\
We can apply the functor $- \otimes \Z/2 $ to the row $q = 0$ and obtain in the columns $p = 0, 1, 2$:
$$
 (\Z/2)^{17} \xleftarrow{\ d^1_{1,0}\ } (\Z/2)^{28}\xleftarrow{\ d^1_{2,0}\ } (\Z/2)^{12}. 
$$
The rest of this row are zeroes. The matrix $d^1_{1,0}$ has rank 16 and the matrix $d^1_{2,0}$ has rank 10. \\
In the rows with $q > 0$, the differential $d^1$ is given by a single arrow $d^1_{1,q}$ from
\\
$E^1_{1,q} \cong (\Homol_q(\Z/2; \thinspace \Z/2))^{13} \oplus (\Homol_q(\Z/3; \thinspace \Z/2))^6 \cong (\Z/2)^{13}$ to
$$
E^1_{0,q} \cong \Homol_q(\Z^2; \thinspace \Z/2) \oplus (\Homol_q(\Z/2; \thinspace \Z/2))^6
\oplus(\Homol_q(\mathcal{D}_2; \thinspace \Z/2))^2 \oplus (\Homol_q(\mathcal{S}_3; \thinspace \Z/2))^4,
$$
and the rest of these rows are zeroes. 
For $q=1$, we have $d^1_{1,1}$ of rank 12 arriving at $E^1_{0,1} \cong (\Z/2)^{16} $. 
For $q \ge 3$, we have $d^1_{1,q}$ of rank 13 arriving at $E^1_{0,q} \cong (\Z/2)^{12+2q} $.
For $q = 2$, we have $d^1_{1,2}$ of rank 13 arriving at $E^1_{0,2} \cong (\Z/2)^{17} $.
The only difficulty in seeing this is to compare the maps from
$\Homol_q(\Z/2; \thinspace \Z/2)$ to $\Homol_q(\mathcal{D}_2; \thinspace \Z/2)$
induced by the different inclusions $\Z/2 \to \mathcal{D}_2$; we use observation \ref{obs} for this purpose.

\subsubsection{The $E^2$-term with $\Z/2$-coefficients}$\text{ }$\\
We obtain in the rows with $q \ge 2$ the $E^2$-term concentrated in the column $p=0$,
\scriptsize
$$
\begin{array}{{l|c}}
q \ge 3 &  \quad({\mathbb{Z} /2)^{2q-1}}\\
q=2 & \quad ({\mathbb{Z} /2})^4, \\
\end{array}
$$
\normalsize
and in the rows $q=0$, $q=1$ it is concentrated in the columns $p=0,1,2$:
\scriptsize
$$ \xymatrix{
q=1 & ({\mathbb{Z} /2)^{4}} & {\mathbb{Z} /2} & 0 \\
q=0 &{\mathbb{Z}}/2 & (\Z/2)^2 & \ar[ull]_{d^2_{2,0} %\txt{ of rank 1}
} (\Z/2)^2. 
}$$
\normalsize
\bigskip

\subsubsection*{The differential $d^2_{2,0}$ with $\Z/2$-coefficients}$\text{ }$\\
The basis $\{ (c,s,c',z), \thinspace F \}$ of $E^2_{2,0}$ with $\Z$-coefficients induces a
basis of $E^2_{2,0}$ with $\Z/2$-coefficients.
The Universal Coefficient Theorem yields an isomorphism from
$\Homol_1(\Gamma_\sigma; \thinspace  \Z) \otimes_\Z \Z/2$ to
$\Homol_1(\Gamma_\sigma; \thinspace  \Z/2) $,
which we will use to transfer the elements $\overline{t_\sigma} \in \Homol_1(\Gamma_\sigma; \thinspace  \Z)$
computed in subsection \ref{d2differential} to $\Homol_1(\Gamma_\sigma; \thinspace  \Z/2) $. \\
For $d^2_{2,0}((c,s,c',z))$ the computation is as follows. As $\overline{t_c}$ generates $ \Homol_1(\Gamma_c; \thinspace  \Z) \cong \Z/2,$ it is transferred to the generator of 
$ \Homol_1(\Gamma_c; \thinspace  \Z/2) \cong \Z/2.$ 
Since $\overline{t_s}$ can be completed with a second element to a
\mbox{$\Z$-basis} of \mbox{$ \Homol_1(\Gamma_s; \thinspace  \Z) \cong \Z^2,$} 
it is transferred to a nontrivial element of $\Homol_1(\Gamma_s; \thinspace  \Z/2) \cong (\Z/2)^2.$ 
The element $\overline{t_z}$ vanishes because $\Homol_1(\Gamma_z; \thinspace  \Z) \otimes \Z/2 \cong \Z/3 \otimes \Z/2 = 0.$ 
The sum $\overline{t_c} + \overline{t_s}$  is quotiented to a nontrivial element on the $E^2$-page 
because $ \Homol_1(\Gamma_s; \thinspace  \Z/2)$ is not hit by the $d^1$-differential. So $d^2_{2,0}(\langle (c,s,c',z) \rangle ) \cong \Z/2 $. \\
For $d^2_{2,0}(F)$, the computation is as follows. Since the 3-torsion vanishes when tensored with $\Z/2$, the 3-torsion part $\bar{t_x} +\bar{t_y}$ of the sum makes no contribution to the image of $d^2$. 
The 2-torsion part, $ \overline{t_b} + \overline{t_a} + \overline{t_v} + \overline{t_u} $,
 equals the image 
$$ d^1_{1,1}( \overline{t_{(b,c)}} + \overline{t_{(c'',a'')}} 
+ \overline{t_{(v,f)}} + \overline{t_{(f,h)}} + \overline{t_{(h,u')}} ), $$
where $ \overline{t_\sigma} $, $ \sigma \in \{b,a,v,u, (b,c),(c'',a''),(v,f),(f,h),(h,u') \}$ %now stands for 
is the generator of 
$ \Homol_1(\Gamma_\sigma; \thinspace  {\mathbb{Z}}/2 ) \cong \Z/2. $
Hence it makes no contribution neither, and we obtain $ d^2(F) = 0.$ Thus $d^2_{2,0}$ has rank 1.
 \medskip \\
As $\Z/2$-modules are vector spaces over the field with two elements $\F_2$, the $E^3=E^\infty$-page yields immediately the results. We do an analogous computation with $\Z/3$-- and $\Z/4$--coefficients and obtain
\begin{center} $ \dim_{\F_2}\Homol_q(\Gamma; \thinspace  \Z/2) = \scriptsize
\begin{cases}
2q-1, & q \ge 3, \\
6, & q = 2, \\
5, &  q= 1,
\end{cases} $  \qquad \normalsize
$ \dim_{\F_3}\Homol_q(\Gamma; \thinspace  \Z/3) = \scriptsize
\begin{cases}
2, & q \equiv 0 \medspace \mathrm{or} \medspace 2 \mod 4, \medspace q > 2,\\
4, & q \equiv 3 \mod 4, \\
0, & q \equiv 1  \mod 4, \medspace q > 2;
\end{cases} $
\end{center}
\normalsize
and the exact sequence $ 1 \rightarrow (\Z/2)^5 \rightarrow \Homol_3(\Gamma; \thinspace  \Z/4) \rightarrow \Z/2 \rightarrow 1. $
%With $\Z/4$-coefficients, we obtain the $E^\infty$-page  concentrated in the columns $p=0,1,2$:
%\scriptsize
%$$
%\begin{array}{{l|ccc}}
%q \ge 3 &  \quad({\mathbb{Z} /2)^{2q-1}} & 0 & 0\\
%q=2 & \quad \Z/4 \oplus ({\mathbb{Z} /2})^2 & \Z/2 & 0 \\
%q=1 & \quad \Z/4 \oplus ({\mathbb{Z} /2})^2 & \Z/2 & 0 \\
%q=0 & \quad \Z/4 & (\Z/4)^2 & \Z/4 \\
%\end{array}
%$$
%\normalsize
The short exact sequence 
$$ 1 \rightarrow \Z \oplus (\Z/2)^2 \rightarrow \Homol_2(\Gamma; \thinspace \Z) 
\rightarrow \Z \oplus (\Z/3)^2 \oplus \Z/2 \rightarrow 1$$
tells us that $ \Homol_2(\Gamma; \thinspace \Z)$  is one of the group extensions \scriptsize 
$ \begin{cases} \Z^2 \oplus (\Z/3)^2 \oplus (\Z/2)^3, \\ \Z^2 \oplus (\Z/3)^2 \oplus (\Z/2)^2, \\
 \Z^2 \oplus \Z/3 \oplus (\Z/2)^3, \\ \Z^2 \oplus (\Z/3)^2 \oplus \Z/2 \oplus \Z/4, \\
 \Z^2 \oplus \Z/3 \oplus \Z/2 \oplus \Z/4. \end{cases}$ \normalsize
\\
Using the Universal Coefficient Theorem in the form
$$ 
\Homol_q(\Gamma; \thinspace  \Z/n) \cong 
\Homol_q(\Gamma; \thinspace \Z) \otimes (\Z/n) \oplus \mathrm{Tor}^\Z_1(\Homol_{q-1}(\Gamma; \thinspace \Z), \Z/n)$$
with $n = 2$, $3$ and $4$, we can now eliminate all the wrong answers and retain
$$ \Homol_q(\text{PSL}_2(\mathcal{O}_{-13}); \thinspace \Z) \cong
\begin{cases}
\Z^3 \oplus (\Z/2)^2, &  q= 1, \\
\Z^2 \oplus \Z/4 \oplus (\Z/3)^2 \oplus \Z/2, & q = 2, \\
(\Z/2)^{q} \oplus (\Z/3)^2, & q = 4k+3, \quad k \ge 0, \\
(\Z/2)^q, & q = 4k+4,  \quad k \ge 0, \\
(\Z/2)^{q}, & q = 4k+1,  \quad k \ge 1, \\
(\Z/2)^{q} \oplus (\Z/3)^2, & q = 4k+2,  \quad k \ge 1.
\end{cases} $$

\subsection{$m = 5$}
We will make use of the following matrices, which agree with those in \cite{Floege}:
\scriptsize \begin{alignat*}8
&A &:=& \pm { \begin{pmatrix}  & -1 \\ 1 & \end{pmatrix} },
& \qquad
&B &:=& \pm { \begin{pmatrix} -\omega & 2 \\ 2 & \omega \end{pmatrix} },\qquad
&M &:=& \pm {\begin{pmatrix} -\omega & 4 \\ 1 & \omega \end{pmatrix} },
& \qquad
&S &:=& \pm {\begin{pmatrix} & -1 \\ 1 & 1 \end{pmatrix} },
\\
&U &:=& \pm {\begin{pmatrix} 1 & \omega \\ & 1 \end{pmatrix} }, & \qquad
&V &:=& \pm { \begin{pmatrix} -\omega & 2-\omega \\ 2 & 2+\omega \end{pmatrix} },\qquad
&W &:=& \pm { \begin{pmatrix} 6-\omega & 5+2\omega \\ 2\omega & \omega-4 \end{pmatrix} }.
\end{alignat*} \normalsize
These are subject to the relations \mbox{$UMU^{-1} = A$,} $UWS(UW)^{-1} = S$, $WABW^{-1} = MB$ and \mbox{$S = ABV$.}
A fundamental domain is displayed in figure \ref{m5}. There are five orbits of vertices, with stabilizers 

\begin{wrapfigure}{r}{25mm}
\includegraphics[width=22mm]{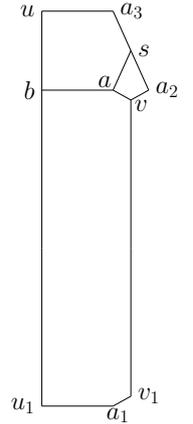}
\caption{The fundamental domain for $m=5$} \label{m5}
\end{wrapfigure}
\scriptsize
$$\begin{array}{*{5}{l}}
\Gamma_{b} &=&\langle A, B| A^2 = B^2 = 1 \rangle \cong {\mathcal{D}}_2,\\
\Gamma_{u} &=&\langle B, M | B^2 = M^2 = 1 \rangle \cong {\mathcal{D}}_2,\\
\Gamma_{a} &=&\langle AB | AB^2 = 1 \rangle \cong \Z/2, \\
\Gamma_{v} &=&\langle S |  S^3 = 1 \rangle \cong \Z/3, \\
\Gamma_{s} &=&\left\langle V,W  | VW = WV \right\rangle \cong \Z^2
\end{array}$$ \normalsize
As in the case $m=13$, vertices with the same letter in the fundamental domain are identified by the action of $\Gamma$. Amongst the identifications of the vertices, we will use the following.\label{identifications5}
 $UW \cdot a = a_1$, $V^{-1} \cdot a = a_2$, $S^2 \cdot a = a_2$, $U \cdot u = u_1$ and $UW \cdot v = v_1$.
There are seven orbits of edges, with stabilizers \scriptsize
$$\begin{array}{*{5}{l}}
\Gamma_{(b,a)} &= &  \langle AB | AB^2 = 1 \rangle \cong \Z/2, \\
\Gamma_{(v,v_1)}&=&  \langle S |  S^3=1 \rangle \cong \Z/3, \\
\Gamma_{(a_3,u)} &=&\langle MB |  MB^2 = 1 \rangle \cong \Z/2, \\
\Gamma_{(u,b)}  &=&\langle B |  B^2 = 1 \rangle \cong \Z/2, \\
\Gamma_{(u_1,b)} &=& \langle A | A^2 = 1\rangle \cong \Z/2;
\end{array}$$ \normalsize
$(a,v)$ and $(a,s)$ having the trivial stabilizer. There are three orbits of faces, with trivial stabilizers.
The above data gives the $\Gamma$-equivariant Euler characteristic of $X$: 
$$ \chi_\Gamma(X) =  \frac{1}{2} +\frac{1}{3} +\frac{2}{4} -2 -\frac{4}{2} -\frac{1}{3} +3 = 0, $$
in accordance with remark \ref{vanish}.

\subsubsection{The bottom row of the $E^1$-term}$\text{ }$\\
This row identifies with the cellular chain complex of the quotient complex $\Gamma\backslash X$. \\
We obtain for the row $q=0$ in the columns $p=0,1,2$:
$$
 {\mathbb{Z}}^5 
\xleftarrow{\ d^1_{1,0}\ } 
 {\mathbb{Z}}^7
\xleftarrow{\ d^1_{2,0}\ } 
{\mathbb{Z}}^3
$$
\mbox{where 1 is the only elementary divisor of the differential matrices, with multiplicity four for $d^1_{1,0}$,}
 and multiplicity two for $d^1_{2,0}$. 
The homology of $\Gamma\backslash X$ is generated in degree 1 by the loop represented by the edge $(v,v_1)$, and in degree 2 by the quotient of the face $(a_2,s,a,v)$, which is homeomorphic to a 2-sphere.

\subsubsection{The odd rows of the $E^1$-term}$\text{ }$\\
We start by investigating the morphism
$$
 \Z^2 \oplus \Z/3 \oplus (\Z/2)^5
\xleftarrow{\ \ d^1_{1,1}\ \ }
\Z/3 \oplus (\Z/2)^4
$$
and the morphism
$$
 \Z/3 \oplus (\Z/2)^{q+4}
\xleftarrow{\ \ d^1_{1,q}\ \ }
 \Z/3 \oplus (\Z/2)^4
$$
for $q \ge 3$.
On the 3-torsion, $d^1_{1,q}$ is zero. \\
\begin{wrapfigure}{r}{60mm}
\caption{ $(d^1_{1,q})_{(2)}$ }
\label{d11q}
 \scriptsize
$\begin{array}{l|cccc}
& (b,a) & (a_3,u)& (u,b) & (u_1,b)
\\
\hline
a  & 1 & -1 & 0  & 0
\\
b &  -1 & 0  & 0  & 1
\\
\vdots   & \vdots &\vdots &\vdots & \vdots
\\
b &  -1 & 0  & 1  & 0
\\
u &   0 & 1  & -1 & 0
\\
\vdots & \vdots     & \vdots & \vdots& \vdots
\\
u &   0 & 1  & 0  & -1,
\end{array}
$ \normalsize
\end{wrapfigure}
On the 2-torsion, $d^1_{1,q}$ is given by the matrix $(d^1_{1,q})_{(2)}$ of figure \ref{d11q},
where we replace the $\frac{q-1}{2}$ dotted entries between the two $1$'s with $1$'s, and the $\frac{q-1}{2}$ dotted entries between the $-1$'s with $-1$'s. The rest is filled with zeroes. Thus $(d^1_{1,1})_{(2)}$ has rank 3 and $(d^1_{1,q})_{(2)}$ has rank 4 for $q \ge 3$.

\subsubsection{The even rows of the $E^1$-term}$\text{ }$\\
There is a zero map arriving at \mbox{$E^1_{0,2} \cong \Z \oplus (\Z/2)^{2}$.} \\
For $q \ge 4$, there is a zero map arriving at $E^1_{0,q} \cong (\Z/2)^{q}$. \\
The rest of the $E^1$-page are zeroes.

\subsubsection{The $E^2$-term}$\text{ }$\\
In the rows with $q \ge 2$, the $E^2$-page is concentrated in the columns \\
$p=0$ and $p=1$:
\scriptsize
$$
\begin{array}{{l|cccc}}
q \ge 4 \text{ even } & (\Z/2)^q  & & 0\\
q \ge 3 \text{ odd } &  (\Z/2)^q \oplus \Z/3 & & \Z/3\\
q = 2 & \Z\oplus (\Z/2)^2 & & 0
\end{array}
$$\normalsize
Its lowest two rows are concentrated in the columns $p=0,1,2$:
$$\scriptsize
\xymatrix{
q=1&\Z^2\oplus (\Z/2)^2 \oplus\Z/3& \Z/2 \oplus \Z/3 & 0\\
q=0&\Z&\Z&\Z  \ar[ull]_{d^2} 
}
$$
\normalsize
\\
Let us compute the only nontrivial $d^2$-arrow.
%$$\Z\quad\longrightarrow\quad\langle \overline{V},\overline{W} \rangle\oplus \langle \overline{S} | 3\overline{S} = 0 \rangle \oplus ({\mathbb{Z}/2})^2$$
The generator of $E^2_{2,0}$ comes from the 2-cell $(a_2,s,a,v)$. Using the identifications listed in \ref{identifications5}, we see that the lift $1 \otimes_{(a_2,s,a,v)}  1 $ of the generator of $E^2_{2,0}$  
is mapped as follows in the $E^0$-page:
\scriptsize
$$\xymatrix{
{
\begin{array}{cc}
-(1,V^{-1}) \otimes_s  1
 +(1,V^{-1}) \otimes_a   1\\
+(1, S^2) \otimes_v   1 
 -(1, S^2) \otimes_a   1
\end{array}
}
& &
\ar[ll]_{1 \otimes \delta} 
{
\begin{array}{cc}
-(1,V^{-1}) \otimes_{(a,s)}  1\\
 +(1, S^2) \otimes_{(a,v)}  1
\end{array}
}
\ar[d]^{d_\Theta \otimes 1} 
&
\\
& &
{
\begin{array}{cc}
1 \otimes_{(a,s)}   1 
-V^{-1} \otimes_{(a,s)}  1\\
 +S^2 \otimes_{(a,v)}   1
-1 \otimes_{(a,v)}   1 
\end{array}
}
& &
1 \otimes_{(a_2,s,a,v)}  1  
\ar[ll]_{1 \otimes \delta} 
}$$
\normalsize
As $S = ABV$, the part lying in $\Theta_1 \otimes_a \Z$ is 
$[V^{-1}] \otimes_a 1 -[S^2] \otimes_a 1 = [S^2 AB] \otimes_a 1 -[S^2] \otimes_a 1$; and goes to
$\overline{S^{-1}S^2 AB} -\overline{S^{-1}S^2} = \overline{AB}$, the generator of $\Homol_1(\Gamma_a; \thinspace  \Z)$.
So, our image in $E^0_{0,1}$ passes to \small
$$ (\overline{V}, 2\overline{S}, \overline{AB})  
\in \langle \overline{V},\overline{W}  \rangle
\oplus \langle \overline{S} \thinspace | \thinspace 3\overline{S} = 0 \rangle
\oplus ({\mathbb{Z}/2})^2 \cong  E^2_{0,1},$$ \normalsize
which is of infinite order and has the following property:
There is no element $\eta \in E^2_{0,1}$ with \\
\mbox{$k\eta = (\overline{V}, 2\overline{S}, \overline{AB})$} for an integer $k>1$. 
So, $$ E^3_{0,1} \cong {\mathbb{Z}} \oplus ({\mathbb{Z}/2})^2 \oplus {\mathbb{Z}/3}. $$
Thus the $E^\infty$-page yields the following short exact sequences:
$$ \scriptsize
\begin{cases}
0 \rightarrow (\Z/2)^q \rightarrow \Homol_q(\Gamma; \thinspace \Z) \rightarrow \Z/3 \rightarrow 0
\quad & \thinspace q \ge 4 \thinspace \mathrm{even},\\
0 \rightarrow \Z/3 \oplus (\Z/2)^q \rightarrow \Homol_q(\Gamma; \thinspace \Z) \rightarrow  0
\quad & \thinspace q \ge 3 \thinspace \thinspace \mathrm{odd},\\
 0 \rightarrow \Z \oplus (\Z/2)^2 \rightarrow \Homol_2(\Gamma; \thinspace \Z) 
\rightarrow \Z/3 \oplus \Z/2 \rightarrow 0,\\
0 \rightarrow {\mathbb{Z}} \oplus {\mathbb{Z} /3} \oplus ({\mathbb{Z} /2)^2} 
\rightarrow \Homol_1(\Gamma; \thinspace \Z) \rightarrow \Z \rightarrow 0.
\end{cases} \normalsize
$$
To resolve the ambiguity at the group extension $ \Homol_2(\Gamma; \thinspace  \Z)$, we compute
\begin{center}$ 
\dim_{\F_2}\Homol_q(\Gamma; \thinspace \Z/2) = $\scriptsize$
\begin{cases}
 4 \qquad  & q = 1,\\
 5 \qquad & q=2,\\ 
 {2q-1} \qquad&q\ge 3 
\end{cases}
$ 
\qquad
$ \Homol_q(\Gamma; \thinspace  \Z/3) \cong (\Z/3)^2$, $q > 2$,
\hfill
$\begin{cases} 1 \rightarrow (\Z/2)^5 \rightarrow \Homol_3(\Gamma ; \thinspace  \Z/4) \rightarrow \Z/2 \rightarrow 1,\\
\\
 1 \rightarrow (\Z/2)^4 \oplus \Z/4 \rightarrow \Homol_2(\Gamma ; \thinspace  \Z/4) \rightarrow \Z/2 \rightarrow 1; 
\end{cases}$
\end{center}
\normalsize where the last two sequences are exact; and get the result
$$
\Homol_q(\text{PSL}_2(\mathcal{O}_{-5}); \thinspace \Z) \cong
\begin{cases}
 \Z^2 \oplus \Z/3 \oplus (\Z/2)^2 \qquad  & q = 1,\\
 \Z \oplus \Z/4 \oplus \Z/3 \oplus \Z/2 \qquad & q=2,\\ 
 \Z/3 \oplus (\Z/2)^q \qquad & q\ge 3. 
\end{cases}
$$
\begin{rem}
For $m=5$, the check introduced in remark \ref{check} takes the following form.\\
The abelianization is $\Gamma^{\text{ab}}\cong \langle\overline{A},\overline{B},\overline{S},\overline{U},\overline{V}: 2\overline{A}=0,2\overline{B}=0,3\overline{S}=0\rangle$. The fundamental group of the quotient space is free, so only the parabolic elements $U$ and $V$ can define nontrivial loops in the quotient space. The element $U$ generates a nontrivial loop, while $V$ generates a trivial loop.
It follows that $E_{0,1}^\infty \cong \Z\oplus(\Z/2)^2\oplus \Z/3$, generated by $\overline{V},\overline{A},\overline{B}$ and $\overline{S}$. This is consistent with the computation above, involving the detailed analysis of the $d^2$-differential.
\end{rem}

\subsection{$m=10$} Let $\omega := \sqrt{-10}$.
We will use the following definitions: \scriptsize \\
$$\begin{array}{llllllllllll}
A&:=& \pm \begin{pmatrix} & -1 \\ 1 & \end{pmatrix}, 
&B&:=& \pm \begin{pmatrix} -\omega & 3 \\ 3 & \omega \end{pmatrix},
&C&:=& \pm \begin{pmatrix} -1 -\omega & 4-\omega \\ 2 & 1+\omega \end{pmatrix},
&D&:=& \pm \begin{pmatrix} \omega -1 & -4 \\ 3 & 1+\omega \end{pmatrix},
\\ \\
L&:=& \pm \begin{pmatrix} \omega & 3 \\ 3 & -\omega \end{pmatrix}, 
&R&:=& \pm \begin{pmatrix}   5+\omega & 2\omega -3 \\ \omega-3 & -4-\omega \end{pmatrix},
&S&:=& \pm \begin{pmatrix}  & -1 \\ 1 & 1 \end{pmatrix}, 
&U&:=& \pm \begin{pmatrix}  1 & \omega \\ & 1 \end{pmatrix},
\\ \\
V&:=& \pm \begin{pmatrix}  1-\omega & 5 \\ 2 & 1+\omega \end{pmatrix},
&W&:=& \pm \begin{pmatrix} 11 & 5\omega \\ 2\omega & -9 \end{pmatrix},
&Y&:=& \pm \begin{pmatrix} \omega -2 & -5 \\ 3 & 2+\omega \end{pmatrix}.
\end{array}$$ \normalsize
\label{identifications10}
Vertices with the same letter are identified by the action. The matrix $U$ acts as a vertical translation by $-\omega$ on the fundamental domain, which is shown in figure \ref{m10}. There are nine orbits of vertices, labelled $a,b,r,u,v,w,x,y,s$. We have the following identifications:
$ UWa = a_1 $, \quad  $ Wa = a_2 $, \quad $Va = a_3 $; \quad 
$S^{-1}v =v_1 $, \quad $ U^{-1}Dv = v_2 $; \quad 
$Dw = w_1 $,\quad  $ U^{-1}Dw = w_2 $; \quad 
$ Db = b_1 $, \quad $ Cb = b_2 $; \quad  $ Dr = r_1$; \quad $ UWx = x_1 $.
On the vertices of $(a,s,a_3,x)$ , we have the identifications $B \cdot a = a_3$ and $V \cdot a = a_3$,
where the matrix $B$ fixes $x$ and the matrix $V$ fixes $s$.
For $(v_1,b_2,r,b,v,w)$, we have the identifications of vertices $ Cb = b_2 $, $Cr = r $,  $S^{2}v =v_1 $ and $S^2w = w $; and we pay particular attention to the matrix $ CR = S^2 AB $ identifying the edge 
$(b,v) \cong (b_2,v_1) $. 
\\
The stabilizers of the vertex orbit representatives are 

$$ \scriptsize
\begin{array}{llllll}
&\Gamma_a &=& \thinspace\Gamma_b &=&\left\langle \left.
R \right| \thinspace
R^3 = 1
\right\rangle \cong \Z/3,
\\
&\Gamma_w &=&&& \left\langle \left.
S \thinspace \right| \thinspace
S^3 = 1
\right\rangle \cong \Z/3,\\
 &\Gamma_y &=&&& \left\langle \left.
A, L 
\right| \thinspace
A^2 = L^2  = (AL)^2 = 1
\right\rangle \cong \mathcal{D}_2,\\
&\Gamma_{u}&=&&& \left\langle \left.
 A, B \thinspace
\right| \thinspace
A^2 = B^2 = (AB)^2 = 1
\right\rangle \cong \mathcal{D}_2,\\
 &\Gamma_r &=&&& \left\langle \left. C
\right| \thinspace
C^2 = 1
\right\rangle \cong \Z/2,
\\  &\Gamma_v &=&&& \left\langle \left. AB
\right| \thinspace
(AB)^2 = 1
\right\rangle \cong \Z/2,
\\
 &\Gamma_{x} &=&&& \left\langle \left.
B \thinspace \right| \thinspace
B^2 = 1
\right\rangle \cong \Z/2,\\
 &\Gamma_s &=&&& \left\langle \left.
V, W
\right| \thinspace
VW = WV
\right\rangle \cong \Z^2.
\end{array}
$$ \normalsize
There are fifteen orbits of edges, labelled $(b,v), (r,w), (b,r), (v,w), (a_2,w_2),\linebreak (y,r_1),(x,a),(u,y),(a,b),(u,v),(a,s),(w,b_1),(r,v_2),(y,x_1),(x,u)$.
\begin{wrapfigure}{r}{25mm}
\includegraphics[width=22mm]{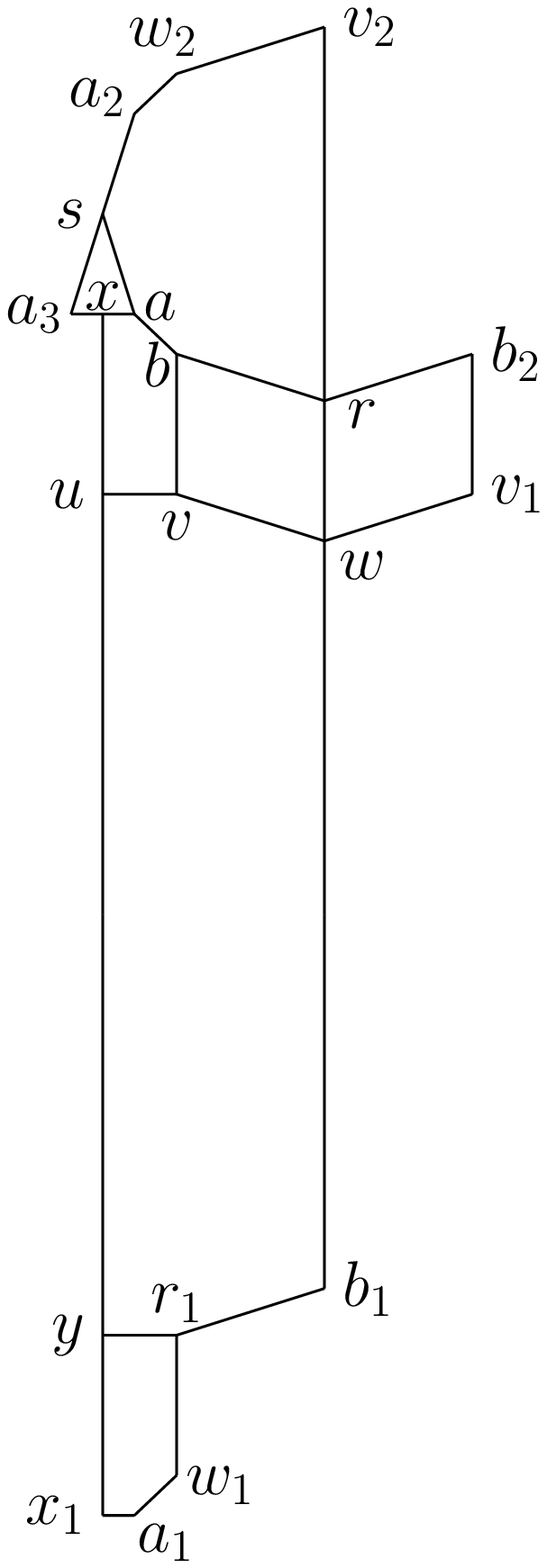} 
\caption{The fundamental domain for $m=10$} \label{m10}
\end{wrapfigure}
Amongst their stabilizers only
\\
$ \scriptsize
\begin{array}{llllllll}
&\Gamma_{(a_2,w_2)}&=&\Gamma_{a_2} &=&&& W^{-1}\Gamma_{a}W = \left\langle \left.
W^{-1}RW \thinspace \right| \thinspace
(W^{-1}RW )^3 = 1
\right\rangle \cong \Z/3,\\
&\Gamma_{(a,b)} &=& \Gamma_a &=& \Gamma_b &=&\left\langle \left.
R  
\right| \thinspace
R^3 = 1
\right\rangle \cong \Z/3,\\
&\Gamma_{(w,b_1)} &=&\Gamma_{b_1} &=&  \Gamma_w &=& \left\langle \left.
S \thinspace \right| \thinspace
S^3 = 1
\right\rangle \cong \Z/3,\\
&\Gamma_{(y,r_1)}&=&\Gamma_{r_1} &=&&& D\Gamma_{r}D^{-1} =\left\langle \left.
 AL=DCD^{-1} \thinspace \right| \thinspace
(DCD^{-1})^2 = 1
\right\rangle \cong \Z/2,\\
&\Gamma_{(u,v)}&=& \Gamma_{v} &=&&& \left\langle \left.
AB \thinspace \right| \thinspace
(AB)^2 = 1
\right\rangle \cong \Z/2,\\
&\Gamma_{(r,v_2)} &=&\Gamma_{v_2} &=& \Gamma_r &=& \left\langle \left. C
\right| \thinspace
C^2 = 1
\right\rangle \cong \Z/2,\\
&\Gamma_{(y,x_1)}&=&\Gamma_{x_1} &=&&& UW\Gamma_{x}(UW)^{-1} =\left\langle \left.
 L \thinspace \right| \thinspace
L^2 = 1
\right\rangle \cong \Z/2,\\
&\Gamma_{(x,u)} &=&\Gamma_{x} &=&&& \left\langle \left.
B \thinspace \right| \thinspace
B^2 = 1
\right\rangle \cong \Z/2,\\
&\Gamma_{(u,y)} &=&&&&&  \left\langle \left.
A \thinspace \right| \thinspace
A^2 = 1
\right\rangle \cong \Z/2
\end{array}
$ \normalsize \\
are nontrivial. Furthermore, there are seven orbits of faces, with trivial stabilizers. 
\\
With the above information on the isomorphism types of the cell stabilizers, we get the $\Gamma$-equivariant Euler characteristic of $X$: 
$$ \chi_\Gamma(X) = 
\frac{3}{3} +\frac{2}{4} +\frac{3}{2} 
-\frac{3}{3} -\frac{6}{2} -6
 +7 
= 0, $$
in accordance with remark \ref{vanish}. 
\subsubsection{The bottom row $q=0$ of the $E^1$-term} % $\text{ }$\\
We obtain for the row $q = 0$ in the columns $p = 0, 1, 2$:
$$
 \Z^9 \xleftarrow{\ d^1_{1,0}\ } \Z^{15} 
\xleftarrow{\ d^1_{2,0}\ } \Z^7,
$$
where 1 is the only elementary divisor of the differential matrices, with multiplicity eight for $d^1_{1,0}$, and multiplicity five for $d^1_{2,0}$. The rest of this row is zero.
\subsubsection{The odd rows of the $E^1$-term}$\text{ }$\\
For odd $q$, the morphism
$$
\bigoplus_{\sigma \in \Gamma\backslash X^0} \Homol_q (\Gamma_\sigma)
\xleftarrow{\ d^1_{1,q}\ }
\bigoplus_{\sigma \in \Gamma\backslash X^1} \Homol_q(\Gamma_\sigma)
$$
is for $q \ge 3 $ of the form
$$
(\Z/3)^3 \oplus (\Z/2)^{q+6}
\xleftarrow{\ \ }
(\Z/3)^3 \oplus (\Z/2)^6.
$$
For $q = 1$, we have to add $\Homol_1(\Gamma_{s}) \cong {\mathbb{Z}^2}$ on the 
target side of the \mbox{morphism $d^1_{1,q}$,} but the incoming torsion must reach it trivially.
On the 3-primary part, $d^1_{1,q}$ is given by the matrix
$$ \scriptsize
(d^1_{1,q})_{(3)}=\begin{array}{c|ccc}
& (a,b) & (Db,w) & (Wa,U^{-1}Dw)
\\
\hline
a & -1 &  0 & -1
\\
w &  0 &  1 & 1
\\
b &  1 & -1 & 0.
\end{array} \normalsize
$$
This matrix has rank 2, so its image is isomorphic to
 $({\mathbb{Z} /3})^2$ 
and its kernel is ${\mathbb{Z} /3}$.
\\ On the 2-primary part, $d^1_{1,q}$ is for odd $q$ given by the matrix 
$$ (d^1_{1,q})_{(2)} = \quad \scriptsize
\begin{array}{c|cccccc}
& (y,r_1) & (u,v) & (r,v_2) & (y,x_1) & (x,u) & (u,y)
\\
\hline
u & 0 & -1 & 0 & 0 & 0 & -1
\\
\vdots  & \vdots & \vdots & \vdots & \vdots & \vdots & \vdots
\\
u & 0 & -1 & 0 & 0 & 1 & 0
\\
y & -1 & 0 & 0 & 0 & 0 & 1
\\
\vdots  & \vdots & \vdots & \vdots & \vdots & \vdots & \vdots
\\
y & -1 & 0 & 0 & -1 & 0 & 0
\\
x & 0 &  0 & 0 & 1 & -1 & 0
\\
r & 1 & 0 & -1 & 0 & 0 & 0
\\
v & 0 & 1 & 1 & 0 & 0 & 0,
\end{array}
$$ \normalsize
where we replace, as in the computation for $m = 13$, the $\frac{q-1}{2}$ dotted entries between the two $1$'s with $1$'s, and the $\frac{q-1}{2}$ dotted entries between the $-1$'s with $-1$'s. The rest is filled with zeroes. The resulting matrix $(d^1_{1,q})_{(2)}$ has rank 5 for $q=1$, and full rank 6 for $q \ge 3$.

\subsubsection{The even rows of the $E^1$-term}$\text{ }$\\
These rows are given by zero maps into $\bigoplus\limits_{\sigma \in \Gamma\backslash X^0} \Homol_q (\Gamma_\sigma) \cong ({\mathbb{Z} /2})^q $  for $q>2$, respectively into \linebreak $\bigoplus\limits_{\sigma \in \Gamma\backslash X^0} \Homol_2 (\Gamma_\sigma) \cong {\mathbb{Z}} \oplus (\Z/2)^2$ for $q=2$.
%\end{center}

\subsubsection{The $E^2$-term}$\text{ }$\\
In the rows with $q \ge 2$, the $E^2$-page is concentrated in the columns $p=0$ and $p=1$:
\scriptsize
$$
\begin{array}{{l|cccc}}
q \ge 4 \text{ even } & (\Z/2)^q  & & 0\\
q \ge 3 \text{ odd } &  (\Z/2)^q \oplus \Z/3 & & \Z/3\\
q=2 & \Z\oplus (\Z/2)^2 & & 0
\end{array}
$$\normalsize
Its lowest two rows are concentrated in the columns $p=0,1,2$:
$$\scriptsize
\xymatrix{
q=1&\Z^2\oplus (\Z/2)^2 \oplus\Z/3& \Z/2 \oplus \Z/3 & 0\\
q=0&\Z&\Z^2 &\Z^2  \ar[ull]_{d^2} 
}
$$
\normalsize

\subsubsection{The differential $d^2$}$\text{ }$\\
The generators of the abelian group $E^2_{2,0} \cong {\mathbb{Z}}^2 $
are represented by the \mbox{2-cell} $(a,s,a_3,x)$ and the
union of two 2-cells $(v_1,b_2,r,b,v,w)$, whose quotients by $\Gamma$
are homeomorphic to 2-spheres. Using the identifications given in \ref{identifications10}, wee that the only nontrivial $d^2$-arrow is induced by
$$ \delta((a,s,a_3,x)) = (a,s) +V \cdot (s,a) +B \cdot (a, x) +(x,a) $$
and
$$ \delta((v_1,b_2,r,b,v,w))= (b,r)- C \cdot (b,r)+CR \cdot (b, v)+ S^2 \cdot (v, w)-(v,w) -(b,v). $$
The lift $1 \otimes_{(v_1,b_2,r,b,v,w)} 1$ of the generator obtained from $(v_1,b_2,r,b,v,w)$ is mapped as follows:\\
\scriptsize
$$\xymatrix@C=2em{
{
\begin{array}{cc}
(C,1) \otimes_{r} 1 
-(C,1) \otimes_{b} 1 \\
+(1,CR) \otimes_{v} 1  
-(1,CR) \otimes_{b} 1 \\
+(1,S^2) \otimes_{w} 1
-(1,S^2) \otimes_{v} 1
\end{array}
}
& &
\ar[ll]_{1 \otimes \delta } 
{
\begin{array}{cc}
 (C,1) \otimes_{(b,r)} 1 \\
 +(1,CR) \otimes_{(b,v)} 1  \\
 +(1,S^2) \otimes_{(v,w)} 1 
\end{array}
}
\ar[d]^{d_\Theta \otimes 1} 
&
\\
& &
{
\begin{array}{cc}
 1 \otimes_{(b,r)} 1
-C \otimes_{(b,r)} 1 \\
 +CR  \otimes_{( b,v)} 1
-1 \otimes_{(b,v)} 1  \\
 +S^2 \otimes_{(v,w)} 1
-1 \otimes_{(v,w)} 1  
\end{array}
}
& &
1 \otimes_{(v_1,b_2,r,b,v,w)} 1  
\ar[ll]_{1 \otimes \delta } 
}$$
\normalsize
We 
%$\Theta_1 \otimes_\sigma \Z $.
%$$ \scriptsize
%\begin{array}{llll}
%&t_r &:=& (C,1) \otimes_r 1,\\
%&t_b &:=& -(1,CR) \otimes_b 1  
%-(C,1) \otimes_b 1,\\
%&t_w &:=& (1,S^2) \otimes_w 1,\\
%&t_v &:=& (1,CR) \otimes_v 1  
%-(1,S^2) \otimes_v 1.
%\end{array}
%$$ \normalsize
%With the formula in our corollary \ref{Corollary}, we find the classes $\overline{t_\sigma}$ in 
%$\Homol_1(\Theta_* \otimes_\sigma  \Z)$ in the following way. 
%$$ \scriptsize
%\begin{array}{llll}
%&\overline{t_r} &=  &\overline{C} \in  
% \langle \overline{C} | 2\overline{C} = 0 \rangle = \Gamma_r^{ab} \cong \Homol_1(\Gamma_r, %%{\mathbb{Z}} ),\\
%&\overline{t_b} &= &2\overline{R} \in 
%\langle \overline{R} | 3\overline{R} = 0 \rangle = \Gamma_b^{\text{ab}} \cong \Homol_1(\Gamma_b, %{\mathbb{Z}} ),\\
%&\overline{t_w} &=  &2\overline{S} \in \langle \overline{S} | 3\overline{S} = 0 \rangle = \Gamma_w^{ab } \cong \Homol_1(\Gamma_w; \Z ),\\
%&\overline{t_v} &=  &\overline{AB} \in 
%\langle \overline{AB} | 2\overline{AB} = 0 \rangle = \Gamma_v^{ab} \cong \Homol_1(\Gamma_v; \Z ).
%\end{array}
%$$ \normalsize
%As $E^1_{-1,1} = 0$, the whole of $E^1_{0,1}$ is in the kernel of $d^1$. 
%So for the passage to $E^2$, we just have to look at the image of $d^1$.
%\\
%The chain $\overline{t_v} +\overline{t_r}$ equals the image $d^1_{1,1}(-(r,v_2))$, so it passes %%%to the zero element in the quotient $ E^2_{0,1}$. Meanwhile,
%the cycle $\overline{t_b} +\overline{t_w}$ is a representative of the generator of the 3-torsion
%in $ E^2_{0,1}$, as we see from the matrix for the 3-primary part of $d^1_{1,1}$.
%Thus, 
obtain $d^2_{2,0}(\langle (v_1,b_2,r,b,v,w) \rangle) \cong \Z/3$.
\\
 
 The lift $1 \otimes_{(a,s,a_3,x)} 1 $ of the generator obtained from $(a,s,a_3,x)$ is mapped\\
\scriptsize
$$\xymatrix{
{
\begin{array}{cc}
(V,1) \otimes_s  1 
 -(V,1) \otimes_a 1 \\  
+(1, B) \otimes_x  1 
 -(1, B) \otimes_a  1
\end{array}
}
& &
\ar[ll]_{1 \otimes \delta } 
{
\begin{array}{cc}
(V,1) \otimes_{(a,s)} 1 \\
+(1,B) \otimes_{(a,x)} 1
\end{array}
}
\ar[d]^{d_\Theta \otimes 1} 
&
\\
& &
{
\begin{array}{cc}
1 \otimes_{(a,s)}  1
-V \otimes_{(a,s)}  1 \\
+B  \otimes_{( a,x)} 1
-1 \otimes_{(a,x)}  1 
\end{array}
}
& &
1 \otimes_{(a,s,a_3,x)} 1  
\ar[ll]_{1 \otimes \delta } 
}
$$
\normalsize
We attribute the symbols $t_\sigma$ to the part of this sum lying in 
$\Theta_1 \otimes_\sigma \Z $,
$$ \scriptsize
\begin{array}{llll}
&t_s &:=& (V,1) \otimes_s  1,\\
&t_x &:=& (1, B) \otimes_x  1,\\
&t_a &:=& -(V,1) \otimes_a  1 -(1, B) \otimes_a 1.
\end{array} \normalsize $$
We find the class 
 $  \overline{t_s} =  -\overline{V} \in \langle \overline{V},\overline{W} \rangle
= \Gamma_s^{\text{ab}} \cong \Homol_1(\Gamma_s; \thinspace  \Z )  \cong {\mathbb{Z}}^2, $
which is a generator of the free part of $E^1_{0,1}$.
It can not be the image of a torsion element from $E^1_{1,1} = (\Z/3)^3 \oplus (\Z/2)^2$.
Therefore, it is preserved when passing from $ E^1_{0,1}$ to $ E^2_{0,1}$. 
The cycles $\overline{t_x}$ and $\overline{t_a}$ are torsion, 
so the fact that $\overline{t_s}$ is a generator of the free part determines that the image $d^2_{2,0}(\langle (a,s,a_3,x)\rangle)$  is of infinite order and has the following property:
There is no element $\eta \in E^2_{0,1} \cong \Z^2 \oplus \Z/3 \oplus (\Z/2)^2$ with $k\eta = d^2_{2,0}(\langle (a,s,a_3,x)\rangle)$  for an integer $k>1$. 
Together with the isomorphism $d^2_{2,0}(\langle (v_1,b_2,r,b,v,w) \rangle) \cong \Z/3$, we obtain
$$ E^3_{0,1} \cong \Z \oplus (\Z/2)^2. $$
Thus the $E^\infty$-page gives the following short exact sequences:
$$\scriptsize
\begin{cases}
 0 \rightarrow (\Z/2)^q \rightarrow \Homol_q(\Gamma; \thinspace \Z) \rightarrow \Z/3 \rightarrow 0
,\quad &\mathrm{for } \thinspace q \ge 4 \thinspace \mathrm{even},\\
0 \rightarrow \Z/3 \oplus (\Z/2)^{q} \rightarrow \Homol_q(\Gamma; \thinspace \Z) \rightarrow  0
,\quad &\mathrm{for } \thinspace q \ge 3 \thinspace \thinspace \mathrm{odd},\\
0 \rightarrow \Z \oplus (\Z/2)^2 \rightarrow \Homol_2(\Gamma; \thinspace \Z) \rightarrow \Z \oplus \Z/3
\oplus \Z/2 \rightarrow 0,\\
0 \rightarrow {\mathbb{Z}} \oplus ({\mathbb{Z} /2)^2} 
\rightarrow \Homol_1(\Gamma; \thinspace \Z) \rightarrow \Z^2 \rightarrow 0.
\end{cases}
$$ \normalsize

Therefore, there is ambiguity in the 3-torsion and the 2-torsion of the short exact sequence for $\Homol_2(\Gamma; \thinspace \Z)$.
To identify the correct group extension, we compute
$$
\dim_{\F_2}\Homol_q(\Gamma; \thinspace  \Z/2) \cong \begin{cases}
{2q-1}, & q \ge 3 \\
6, & q = 2, \\
5, & q = 1.
\end{cases}
$$  
Furthermore, we find $\Homol_q(\Gamma; \thinspace  \Z/3) \cong 
(\Z/3)^2
$ for all $q \geq 3$ and the exact sequence \\
$1 \rightarrow (\Z/2)^5 \rightarrow \Homol_3(\Gamma; \thinspace  \Z/4) \rightarrow \Z/2 \rightarrow 1.$

From here, we easily see the results, 
$$
\Homol_q(\text{PSL}_2(\mathcal{O}_{-10}); \thinspace \Z) \cong
\begin{cases}
\Z^3  \oplus (\Z/2)^2, & q=1,\\
\Z^2 \oplus \Z/4 \oplus \Z/3 \oplus \Z/2, & q=2,\\
\Z/3 \oplus (\Z/2)^{q}, & q \ge 3;
\end{cases}
$$

\begin{rem}
For $m=10$, the check introduced in remark \ref{check} takes the following form. The abelianization is the group
$\Gamma^{\text{ab}}\cong \langle\overline{A},\overline{B},\overline{D},\overline{U},\overline{W}:2\overline{A}=2\overline{B}=0\rangle.$
The elements of infinite order are $D$, $U$ and $W$. The elements $U$ and $U^{-1}D$ give the cycles generating $\Homol_1(\Gamma \backslash X)$, while $W$ generates a trivial loop.
So it follows that $E_{0,1}^\infty=\Z\oplus (\Z/2)^2$, generated by $\overline{W},\overline{A}$ and $\overline{B}$. This is consistent with the computation above.
\end{rem}

\subsection{$m=6$}
We obtain the fundamental domain  for $\Gamma = \text{PSL}_2(\Z[\sqrt{-6}\thinspace])$ displayed in figure \ref{m6}. The matrix $U := \pm \left( \begin{smallmatrix}1 & \omega \\ & 1
\end{smallmatrix}
\right)$ performs a vertical translation by $-\omega$ of the fundamental domain. The following matrices occur in the cell stabilizers.
\scriptsize $$\begin{array}{lllllllllllllllllllllllllllllll}
\\
A &:=& \pm { \begin{pmatrix}  & -1 \\ 1 & \end{pmatrix} },
&B &:=& \pm { \begin{pmatrix} -1-\omega & 2-\omega \\ 2 & 1+\omega \end{pmatrix} },
&R &:=& \pm {\begin{pmatrix} -\omega & 5-\omega \\ 1 & 1+\omega \end{pmatrix} },
\\ \\
S &:=& \pm {\begin{pmatrix} & -1 \\ 1 & 1 \end{pmatrix} },
&V &:=& \pm { \begin{pmatrix} 1-\omega & 3 \\ 3 & 1+\omega \end{pmatrix} },
&W &:=& \pm { \begin{pmatrix} 7 & 3\omega \\ 2\omega & -5 \end{pmatrix} }. \\
\end{array}$$ \normalsize 
There are five orbits of vertices, labelled $b,a,u,v,s$, with stabilizers\\ 
\begin{wrapfigure}{r}{25mm}
\includegraphics[width=22mm]{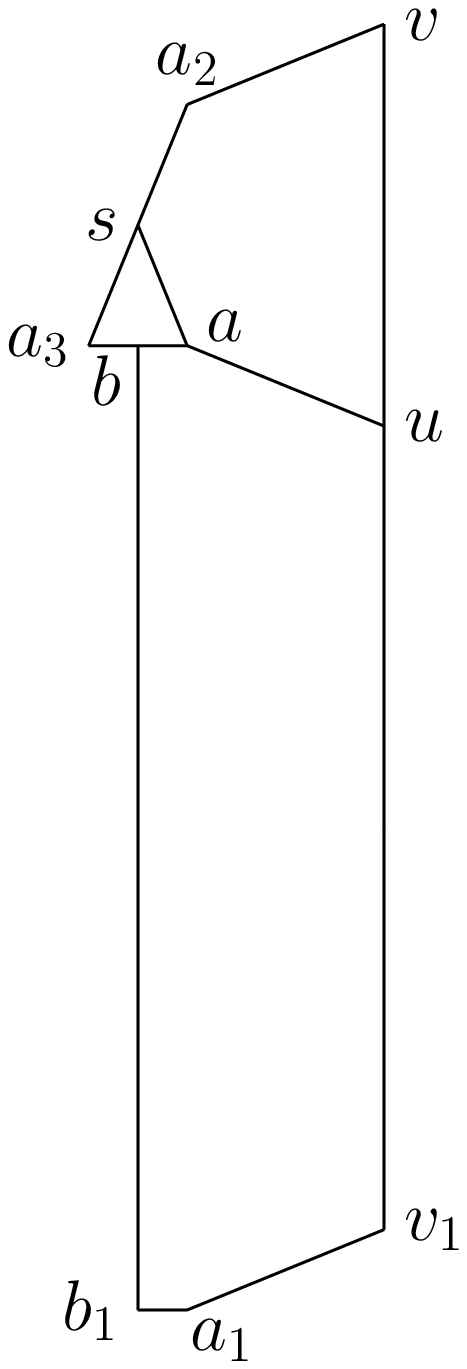}
\caption{The fundamental domain for $m=6$} \label{m6}
\end{wrapfigure}
$$ \scriptsize
\begin{array}{llll}
&\Gamma_u &=& \left\langle \left.
B,S \right| \thinspace
B^2 = S^3 = (BS)^3 = 1
\right\rangle \cong \mathcal{A}_4,\\
&\Gamma_v &=& \left\langle \left.
B, \quad
R
\right| \thinspace
B^2 = R^3 = (BR)^3 = 1
\right\rangle \cong \mathcal{A}_4,\\
&\Gamma_a &=& \left\langle \left.
SB \thinspace \right| \thinspace
(SB)^3 = 1
\right\rangle \cong \Z/3,\\
&\Gamma_b &=& \left\langle \left.
A
\right| \thinspace
A^2 = 1
\right\rangle \cong \Z/2,\\
&\Gamma_s &=& \left\langle \left.
V,W\right| \thinspace
VW = WV
\right\rangle \cong \Z^2,
\end{array}
$$ \normalsize
\label{identifications6}
and identifications $UW \cdot a = a_1$, $W \cdot a = a_2$, $V \cdot a = a_3$, $A \cdot a = a_3$, $UW \cdot b = b_1$ and \mbox{$U \cdot v = v_1$.}
There are seven orbits of edges, labelled $(b,a)$, $(a,s)$, $(a,u)$, $(u,v)$, $(a_2,v)$, $(b,b_1)$ and $(u,v_1)$, amongst whose stabilizers  only
$$ \scriptsize
\begin{array}{llll}
&\Gamma_{(a_2,v)}&=& \left\langle \left.
RB \thinspace \right| \thinspace
(RB)^3 = 1
\right\rangle = \Gamma_{a_2} \cong \Z/3,\\
&\Gamma_{(u,v_1)}&=& \left\langle \left.
S \thinspace \right| \thinspace
S^3 = 1
\right\rangle \cong \Z/3,\\
&\Gamma_{(a,u)} &=& \left\langle \left.
SB \thinspace \right| \thinspace
(SB)^3 = 1
\right\rangle = \Gamma_a \cong \Z/3,\\
&\Gamma_{(u,v)}&=& \left\langle \left.
B \thinspace \right| \thinspace
B^2 = 1
\right\rangle \cong \Z/2,\\
&\Gamma_{(b,b_1)}&=& \left\langle \left.
A \thinspace \right| \thinspace
A^2 = 1
\right\rangle = \Gamma_b = \Gamma_{b_1} \cong \Z/2
\end{array}
$$ \normalsize
are nontrivial; and three orbits of faces with trivial stabilizers.
The above data gives the $\Gamma$-equivariant Euler characteristic of $X$: 
$$ \chi_\Gamma(X) = \frac{2}{12} +\frac{1}{3} +\frac{1}{2} -2 -\frac{3}{3}
-\frac{2}{2} +3 = 0, $$
in accordance with remark \ref{vanish}. 

\subsubsection{The bottom row of the $E^1$-term}$\text{ }$\\
We obtain in the columns $p=0,1,2$:
$$
 {\mathbb{Z}}^5
\xleftarrow{\ d^1_{1,0}\ }
 {\mathbb{Z}}^7
\xleftarrow{\ d^1_{2,0}\ } 
{\mathbb{Z}}^3
$$
where 1 is the only occurring elementary divisor of the differential matrices, with multiplicity four for $d^1_{1,0}$, and multiplicity two for $d^1_{2,0}$. 
The homology of this sequence is generated by the cycle $(b,b_1)$ in degree one and by the face $(a,s,a_3,b)$ in degree two. 
% As in the case $m=5$, a comparison with \cite{Vogtmann} already shows that the $d^2$-differential must be non-zero (see remark \ref{remark}).

\subsubsection{The odd rows of the $E^1$-term}$\text{ }$\\
The map 
$d^1_{1,q}$ is on the 2-primary part induced by the inclusion of  
\mbox{$\Gamma_{(u,v)} \cong \Z/2$}
into $\Gamma_v$ and $\Gamma_u$ which are isomorphic to $\mathcal{A}_4$.
By \cite{SchwermerVogtmann}*{lemma 4.5(2)}, every inclusion of $\Z/2$ into $\mathcal{A}_4$ induces injections on homology in degrees greater than 1, and is zero on $\Homol_1$.
So the morphism
$$
\Z^2 \oplus \Z/2 \oplus (\Z/3)^3 
\xleftarrow{\ d^1_{1,1} \ }
 (\Z/2)^2\oplus(\Z/3)^3 
$$
has $\Z/2$-rank 0 on the 2-primary part, and 
$$
\Z/3 \oplus \Z/2 \oplus (\Homol_q(\mathcal{A}_4))^2
\xleftarrow{\ d^1_{1,q}\ }
 (\Z/2)^2\oplus(\Z/3)^3 
$$
in the odd rows of degree $q\ge 3$ has $\Z/2$-rank 1 on the 2-primary part.
 \\
\paragraph{On the 3-primary part, $d^1_{1,q}$ is for all odd $q$ given by the following rank 2 matrix}
$$ 
(d^1_{1,q})_{(3)}= \quad \scriptsize
\begin{array}{l|ccc}
 & (a,u) & (a_2,v) & (u,v_1)\\
\hline
a  & -1 & -1 & 0
\\
u &  1 & 0  & -1
\\
v &  0 & 1  & 1.
\end{array} \normalsize
$$
 In order to determine its rank, we make use of the following facts.\\
 First, by \cite{SchwermerVogtmann}*{lemma 4.5}, each of the occurring group inclusions induces an injection in homology. So we have to determine the relative positions of the images coming from the edges in each direct summand over the points. In order to find out if cancellation occurs between terms with positive and negative signs, let us look at the following diagram. The symbol $\Delta W$ denotes the isomorphism given by conjugation with $W$, $\delta$ denotes an inner automorphism, $\iota$ denotes any canonical inclusion, and the arrows emanating from $\Z/3$ are labeled with the image of the canonical generator.
$$
 \xymatrix@C=3em{
&
&
{
\Gamma_{(a_2,v)}
}
\ar[lld]_{\mathrm{id}} \ar[rrd]^{\iota}
\\
 {
 \Gamma_{a_2}
 }
 \ar[d]_{\Delta W}
 &&
 \Z/3\ar[u]_{RB}\ar[ll]^{RB}\ar[rr]_{RB}\ar[dll]^{SB}\ar[drr]_>>>>>>>{SU\!BU^{-1}}\ar[ddll]^{SB}\ar[ddr]_{S}\ar[ddl]^{SB}\ar[ddrr]_{S}\ar[dd]^{S}
&&{
 \Gamma_v
 }
\ar[d]^{\Delta U}
\\
 {
 \Gamma_a
 }
&&
&&
{
\Gamma_{v_1}
 }
\\
 {
 \Gamma_{(a,u)}
 }
\ar[u]^{\mathrm{id}} 
\ar[r]_{\iota}
&
{
\Gamma_{u}
}
\ar[r]_\delta
 &
{
\Gamma_{u}
}
&
{
\Gamma_{(u,v_1)}
}
\ar[r]_{\iota}\ar[l]^{\iota}
&
\ar[u]_\delta
{\Gamma_{v_1}
}
}
$$
\normalsize
Applying homology $\Homol_q$ for odd $q$ and taking into account that the fact that inner automorphisms act trivially on homology, 
we get a similar slightly smaller commutative diagram. 
One can then unambiguously identify all occurring groups $\Homol_q(\Z/3)\cong\Z/3$ 
and its images in $\Homol_q(\mathcal{A}_4)$ with the ``abstract'' $\Homol_q(\Z/3)\cong \Z/3$ in the middle. 
This gives a basis for the 3-primary parts of the source and a subspace of the image. 
In this basis, the 3-primary map is given by the above matrix $(d^1_{1,q})_{(3)}$, followed by an injection which does not influence the homology.

\subsubsection{The even rows of the $E^1$-term}$\text{ }$\\
The even rows are the zero map
to  $E^1_{0,2} \cong \Z \oplus (\Z/2)^2,$ and to 
\mbox{$ E^1_{0,q}  \cong (\Homol_q(\mathcal{A}_4) )^2 $}
for degree $q\ge 4$.

\subsubsection{The $E^2$-term}$\text{ }$ \\ 
In the rows with $q \ge 2$, the $E^2$-page is concentrated in the columns $p=0$ and $p=1$:
\scriptsize
$$
\begin{array}{{l|cccc}}
q=6k+8 & (\Z/2)^{2k+4}& & 0\\
q=6k+7 & (\Z/2)^{2k+2} \oplus \Z/3 & & \Z/2 \oplus \Z/3 \\
q=6k+6 & (\Z/2)^{2k+2}& & 0\\
q=6k+5 & (\Z/2)^{2k+4} \oplus \Z/3 & & \Z/2 \oplus \Z/3 \\
q=6k+4 & (\Z/2)^{2k}& & 0\\
q=6k+3 & (\Z/2)^{2k+2}\oplus \Z/3 & & \Z/2 \oplus \Z/3 \\
q=2 & \Z\oplus (\Z/2)^2 & & 0
\end{array}
$$\normalsize
Its lowest two rows are concentrated in the columns $p=0,1,2$:
$$\scriptsize
\xymatrix{
q=1&\Z^2\oplus \Z/2\oplus\Z/3&(\Z/2)^2\oplus\Z/3&0\\
q=0&\Z&\Z&\Z  \ar[ull] 
}
$$
\normalsize

\subsubsection{The $E^3=E^\infty$-term}$\text{ }$\\
For the calculation of the $d^2$-differential, we have
\small \begin{align*}
\delta(a,s,a_3,b) &= (a_3,s)+(s,a)+(a,b)+(b,a_3)\\
&= (V\cdot a,s)+(s,a)+(a,b)+(b,A \cdot a)\\
&= V \cdot (a,s)-(a,s)-(b,a)+A \cdot (b,a),\\
\\
(1\otimes\delta)(1\otimes_{(a,s,a_3,b)} 1) 
&= 1\otimes_{V \cdot (a,s)} 1 -1\otimes_{(a,s)} 1 -1\otimes_{(b,a)} 1 +1\otimes_{A \cdot (b,a)} 1\\
&=(V-1)\otimes_{(a,s)} 1+(A-1)\otimes_{(b,a)} 1\\
&=(d_\Theta\otimes 1)\left((1,V)\otimes_{(a,s)} 1+(1,A)\otimes_{(b,a)} 1 \right) \\
&=(d_\Theta\otimes 1)\left([V]\otimes_{(a,s)} 1+[A]\otimes_{(b,a)} 1 \right).
\end{align*} \normalsize
We then get \small
$$ 
(1\otimes\delta)\left([V]\otimes_{(a,s)} 1+[A]\otimes_{(b,a)} 1 \right)
= [V]\otimes_s 1 -[V]\otimes_a 1 
+[A]\otimes_a 1 -[A]\otimes_b 1. 
$$  \normalsize
As $[V]\otimes_s 1$ and $[W]\otimes_s 1$ represent the generators of the torsion-free part of $E^2_{0,1}\cong\Z^2\oplus\Z/2\oplus\Z/3$, we see that the above computed element of $E^0_{0,1}$ represents an element $\nu \in E^2_{0,1}$ of infinite order with the following property:
there is no element $\eta \in E^2_{0,1}$ with $k\eta = \nu$ for an integer $k>1$. So,
$E^3_{0,1} \cong \Z\oplus \Z/3\oplus\Z/2$ and $E^3_{2,0} = 0$.

\subsubsection{The short exact sequences}$\text{ }$\\
We thus obtain for integral homology the following short exact sequences:
$$ \scriptsize
\begin{cases}
0 \rightarrow (\Z/2)^{2k+4} \rightarrow \Homol_q(\Gamma; \thinspace \Z) \rightarrow \Z/3 \oplus \Z/2\rightarrow 0 ,\quad & \thinspace q = 6k+8
\\
0 \rightarrow (\Z/2)^{2k+2} \oplus \Z/3   \rightarrow \Homol_q(\Gamma; \thinspace \Z) \rightarrow  0
,\quad & \thinspace q = 6k + 7 
\\
 0 \rightarrow (\Z/2)^{2k+2} \rightarrow \Homol_q(\Gamma; \thinspace \Z) \rightarrow \Z/3 \oplus \Z/2 \rightarrow 0
,\quad &  \thinspace q = 6k+6, 
\\
 0 \rightarrow (\Z/2)^{2k+4} \oplus \Z/3   \rightarrow \Homol_q(\Gamma; \thinspace \Z) \rightarrow  0
,\quad &  \thinspace q = 6k +5, 
\\
 0 \rightarrow (\Z/2)^{2k} \rightarrow \Homol_q(\Gamma; \thinspace \Z) \rightarrow \Z/3 \oplus \Z/2 \rightarrow 0
,\quad &  \thinspace q = 6k+4, 
\\
 0 \rightarrow (\Z/2)^{2k+2} \oplus \Z/3   \rightarrow \Homol_q(\Gamma; \thinspace \Z) \rightarrow  0
,\quad &  \thinspace q = 6k +3, 
\\
 0 \rightarrow \Z \oplus (\Z/2)^2 \rightarrow \Homol_2(\Gamma; \thinspace \Z) 
\rightarrow \Z/3 \oplus (\Z/2)^2\rightarrow 0,
\\
 0 \rightarrow \Z\oplus \Z/3\oplus\Z/2 \rightarrow \Homol_1(\Gamma; \thinspace \Z) \rightarrow \Z \rightarrow 0.
\end{cases} \normalsize
$$

Thus, there is ambiguity similar to the case $m=10$ in the 3-torsion of the short exact sequence for $\Homol_2(\Gamma; \thinspace \Z)$ and in the 2-torsion for all even degrees. To resolve it, we compute 
\begin{center}
$\dim_{\F_2}\Homol_q(\Gamma; \thinspace  \Z/2) \cong $\scriptsize$
\begin{cases}
3, & q=1,\\
5, & q=2,\\
{4k+5}, & q = 6k+3,\\
{4k+3}, & q = 6k+4,\\
{4k+5}, & q = 6k +5,\\
{4k+7}, & q = 6k+6,\\
{4k+5}, & q = 6k +7\\
{4k+7}, & q = 6k+ 8,
\end{cases} 
$\small \hfill
\begin{tabular}{l}
$\Homol_q(\Gamma; \thinspace  \Z/3) \cong 
(\Z/3)^2
$ for all $q \geq 3$, \\ \\ and the exact sequences \\ \\
$\begin{cases}
 1 \rightarrow (\Z/2)^4   \rightarrow \Homol_3(\Gamma; \thinspace \Z/4) \rightarrow (\Z/2)^2  \rightarrow 1, \\ \\
 1 \rightarrow \Z/4 \oplus (\Z/2)^3 \rightarrow \Homol_2(\Gamma; \thinspace \Z/4) \rightarrow (\Z/2)^2\rightarrow 1.
\end{cases}$
\end{tabular}
\end{center}
\normalsize

Summarizing, we have resolved the ambiguities and obtain:
$$
\Homol_q(\text{PSL}_2(\mathcal{O}_{-6}); \thinspace \Z) \cong
\begin{cases}
\Z^2 \oplus \Z/3 \oplus \Z/2, & q=1,\\
\Z \oplus \Z/4 \oplus \Z/3 \oplus (\Z/2)^2, & q=2,\\
\Z/3 \oplus (\Z/2)^{2k+2}, &  q = 6k +3,\\
\Z/3 \oplus (\Z/2)^{2k+1}, &   q = 6k+4,\\
\Z/3 \oplus (\Z/2)^{2k+4}, &  q = 6k +5,\\
\Z/3 \oplus (\Z/2)^{2k+3}, &   q = 6k+6,\\
\Z/3 \oplus (\Z/2)^{2k+2}, &  q = 6k + 7,\\
\Z/3 \oplus (\Z/2)^{2k+5}, &   q = 6k+8, \thinspace q \ge 8.
\end{cases}
$$

\begin{rem}
For $m=6$, the check introduced in remark \ref{check} takes the following form. The abelianization is $\Gamma^{\text{ab}}\cong \langle\overline{A},\overline{R},\overline{U},\overline{W}: 2\overline{A}=0,3\overline{R}=0\rangle$. The parabolic element $U$ gives the cycle generating $\Homol_1(\Gamma \backslash X)$, while the parabolic element $W$ generates a trivial loop in the quotient space.
So it follows that $E_{0,1}^\infty \cong \Z\oplus \Z/2\oplus \Z/3$, generated by $\overline{W},\overline{A}$ and $\overline{R}$. This is consistent with the computation above.
\end{rem}

\subsection{$m=15$}

We have $\mathcal{O}_{\rationals\left[\sqrt{-15}\thinspace\right]} = \Z[\omega]$ with 
$\omega := -\frac{1}{2} +\frac{1}{2}\sqrt{-15}$. 
Writing $\Gamma:=\PSL_2(\Z[\omega])$ and 
\scriptsize $$\begin{array}{lllllllllll}
A :=& \pm { \begin{pmatrix}  & -1 \\ 1 & \end{pmatrix} },
&C :=& \pm \begin{pmatrix} 4 & -1 -2\omega \\ 1 +2\omega & 4 \end{pmatrix},
&T :=& \pm \begin{pmatrix} -3 +\omega & -3 -2\omega \\ -1 -2\omega & 4 \end{pmatrix} ,
\\ \\
U :=& \pm  \begin{pmatrix}1 & 1+\omega \\ & 1 \end{pmatrix},
&V :=& \pm { \begin{pmatrix} -1 -2\omega & 3 -\omega \\ 4 & 3 +2\omega \end{pmatrix} },
&W :=& \pm { \begin{pmatrix} -1 -2\omega & 4 \\ 4 +\omega & -1 +2\omega \end{pmatrix} },
& S :=& \pm {\begin{pmatrix} & -1 \\ 1 & 1 \end{pmatrix} },
\end{array}$$ \normalsize
\begin{wrapfigure}{r}{25mm}
\includegraphics[width=22mm]{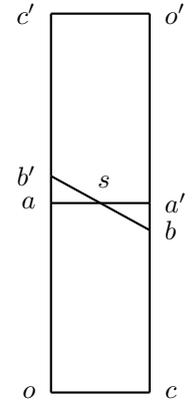}
\caption{The fundamental domain for $m=15$} \label{m15}
\end{wrapfigure}
\label{identifications15}
we have the identifications 
$U^{-1}A \cdot (o,c) = (o',c')$, \quad $T\cdot (a,b') = (a',b)$,
\mbox{$ W \cdot (s,b') = (s,b)$,} and $ V^{-1} \cdot (s,a) = (s,a') $ in the fundamental domain displayed in figure \ref{m15}.
There is no identification bet\-ween the edges $(b,c)$ and $(b',c')$, nor between the edges $(a,o)$ and $(a',o')$. Thus the quotient by the $\Gamma$-action is homeomorphic to the sum of a M\"obius band and a 2-sphere, with a disk amalgamated.
There are five orbits of vertices, labelled $o,a,b,c,s$, with stabilizers 
$$ \scriptsize
\begin{array}{llllllllll}
&\Gamma_o & = & \Gamma_a &=& \left\langle \left.
A \thinspace \right| \thinspace
A^2 = 1
\right\rangle & \cong & \Z/2,\\
&\Gamma_c & = &\Gamma_{b} &=& \left\langle \left.
S
\thinspace \right| \thinspace
S^3 = 1
\right\rangle & \cong & \Z/3,\\
& &  & \Gamma_s &=& \left\langle \left.
V,W \thinspace \right| \thinspace
VW = WV
\right\rangle & \cong & \Z^2.
\end{array}
$$ \normalsize
There are eight orbits of edges, labelled $(o,a)$, $(o',a')$,$(a,s)$, $(a,b')$, $(b,s)$, $(b,c)$, $(b',c')$ and $(o,c)$, amongst whose stabilizers  only
\\ \scriptsize
$\begin{array}{llllllllll}
&\Gamma_{(o,a)}&=& \left\langle \left.
A \thinspace \right| \thinspace
A^2 = 1
\right\rangle &=& \Gamma_o &=& \Gamma_a & \cong & \Z/2,\\
&\Gamma_{(o',a')}&=& \left\langle \left.
V^{-1}AV \thinspace \right| \thinspace
(V^{-1}AV)^2 = 1
\right\rangle &=& \Gamma_{o'} &=& \Gamma_{a'} & \cong & \Z/2,\\
&\Gamma_{(b,c)}&=& \left\langle \left.
S \thinspace \right| \thinspace
S^3 = 1
\right\rangle &=& \Gamma_b &=& \Gamma_c & \cong & \Z/3,\\
&\Gamma_{(b',c')} &=& \left\langle \left.
U^{-1}ASA^{-1}U \thinspace \right| \thinspace
(U^{-1}ASA^{-1}U )^3 = 1
\right\rangle & = & \Gamma_{b'}  & = & \Gamma_{c'} & \cong & \Z/3 \\
\end{array}
$ \normalsize \\
are nontrivial; and four orbits of faces with trivial stabilizers. 
The above data gives the $\Gamma$-equivariant Euler characteristic of $X$, in accordance with remark \ref{vanish}: 
$$ \chi_\Gamma(X) = \frac{2}{2} +\frac{2}{3} -4 -\frac{2}{2} -\frac{2}{3} +4 = 0. $$ 

\subsubsection{The bottom row of the $E^1$-term}$\text{ }$\\
We obtain in the columns $p=0,1,2$:
$$
 {\mathbb{Z}}^5
\xleftarrow{\ d^1_{1,0}\ }
 {\mathbb{Z}}^8
\xleftarrow{\ d^1_{2,0}\ } 
{\mathbb{Z}}^4
$$
where 1 is the only occurring elementary divisor of the differential matrices, with multiplicity four for $d^1_{1,0}$, and multiplicity three for $d^1_{2,0}$. 
The homology of this sequence is generated by the cycle $(o,a)+(a,b')+(b',c')+(c',o')$ in degree one and by the cycle $(a,s,b')-(a',s,b)$ in degree two. 
% As in the case $m=5$, a comparison with \cite{Vogtmann} already shows that the $d^2$-differential must be non-zero (see remark \ref{remark}).

\subsubsection{The odd rows of the $E^1$-term}$\text{ }$\\
The maps 
$$ (\Z/2)^2 \oplus (\Z/3)^2 \xleftarrow{\ d^1_{1,q}\ } (\Z/2)^2 \oplus (\Z/3)^2 $$
for $ q \ge 3$, and 
$$\Z^2 \oplus (\Z/2)^2 \oplus (\Z/3)^2 \xleftarrow{\ d^1_{1,1}\ } (\Z/2)^2 \oplus (\Z/3)^2 $$ 
are on the 2-primary part induced by the identity maps 
\mbox{$\Gamma_{(o,a)}= \Gamma_o = \Gamma_a $} and
$\Gamma_{(o',a')}= \Gamma_{o'} = \Gamma_{a'}$. So, 
we obtain the following rank 1 matrix for the 2-primary part: $$ \scriptsize
(d^1_{1,q})_{(2)}=\begin{array}{c|cc}
& (o,a) & (o',a') 
\\
\hline
a & -1 &  -1
\\
o &  1 &   1 
\end{array} \normalsize .
$$
On the 3-primary part, they are induced by the identity maps 
\mbox{$\Gamma_{(b,c)}= \Gamma_b = \Gamma_c $} and
$\Gamma_{(b',c')}= \Gamma_{b'} = \Gamma_{c'}$. So, 
we obtain the following rank 1 matrix for the 3-primary part: $$ \scriptsize
(d^1_{1,q})_{(3)}=\begin{array}{c|cc}
& (b,c) & (b',c') 
\\
\hline
b & -1 &  -1
\\
c &  1 &   1 
\end{array} \normalsize .
$$

\subsubsection{The even rows of the $E^1$-term}$\text{ }$\\
The even rows are the zero map
to  $E^1_{0,2} \cong \Z $, and to 
\mbox{$ E^1_{0,q}  = 0 $}
for $q\ge 4$. 

\subsubsection{The $E^2$-term}$\text{ }$ \\ 
In the rows with $q \ge 2$, the $E^2$-page is concentrated in the columns $p=0$ and $p=1$:
\scriptsize
$$
\begin{array}{{l|cccc}}
q \ge 4 \quad \mathrm{ even} & 0 & & 0\\
q \ge 3 \quad \mathrm{ odd}  & \Z/2 \oplus \Z/3 & & \Z/2 \oplus \Z/3 \\
q = 2 & \Z & & 0 \\
\end{array}
$$\normalsize
Its lowest two rows are concentrated in the columns $p=0,1,2$:
$$\scriptsize
\xymatrix{
q=1&\Z^2\oplus \Z/2\oplus\Z/3& \Z/2 \oplus \Z/3 & 0\\
q=0 & \Z & \Z & \Z  \ar[ull]_{d^2_{2,0}} 
}
$$
\normalsize

\subsubsection{The $E^3=E^\infty$-term}$\text{ }$\\
For the calculation of the $d^2$-differential, we have
\scriptsize \begin{align*}
\delta\left((a,s,b')-(a',s,b)\right) &= (a,s)+(s,b')+(b',a)-(a',s)-(s,b)-(b,a')\\
 &= (a,s) +W^{-1} \cdot (s,b)+(b',a) -V^{-1} \cdot (a,s)-(s,b) -T \cdot (b', a),\\
\\
(1\otimes\delta)(1\otimes_{(a,s,b')-(a',s,b)} 1) 
&= -(V^{-1}-1)\otimes_{(a,s)}1 +(W^{-1} -1)\otimes_{(s,b)}1 -(T-1)\otimes_{(b',a)}1\\
&=(d_\Theta \otimes 1)\left(-(1,V^{-1})\otimes_{(a,s)}1 +(1,W^{-1})\otimes_{(s,b)}1 -(1,T)\otimes_{(b',a)}1 \right)\\
&=(d_\Theta \otimes 1)\left(-[V^{-1}]\otimes_{(a,s)}1 +[W^{-1}]\otimes_{(s,b)}1 -[T]\otimes_{(b',a)}1 \right).
\end{align*} \normalsize
We then get \scriptsize \begin{multline*}
1\otimes\delta\left(-[V^{-1}]\otimes_{(a,s)}1 +[W^{-1}]\otimes_{(s,b)}1 -[T]\otimes_{(b',a)}1 \right) 
= [V^{-1}]\otimes_{a}1 -[V^{-1}]\otimes_{s}1  +[W^{-1}]\otimes_{b}1 -[W^{-1}]\otimes_{s}1 +[T]\otimes_{b'}1 -[T]\otimes_{a}1. 
\end{multline*}  \normalsize
As the generators of the torsion-free part of \mbox{
$E^2_{0,1}\cong\Z^2\oplus\Z/2\oplus\Z/3$} are represented by $-[V^{-1}]\otimes_s 1$ and $-[W^{-1}]\otimes_s 1$,  we see that the above computed element of $E^0_{0,1}$ represents an element $\nu \in E^2_{0,1}$ of infinite order with the following property:
There is no element $\eta \in E^2_{0,1}$ with $k\eta = \nu$ for an integer $k>1$. So,
$E^3_{0,1} \cong \Z\oplus \Z/3\oplus\Z/2$ and $E^3_{2,0} = 0$.

\subsubsection{The short exact sequences}$\text{ }$\\
We thus obtain for integral homology the following short exact sequences:
$$ \scriptsize
\begin{cases}
 0 \rightarrow \Z/2 \oplus \Z/3 \rightarrow \Homol_q(\Gamma; \thinspace \Z) \rightarrow  0
,\quad &  \thinspace q \ge 3, 
\\
 0 \rightarrow \Z  \rightarrow \Homol_2(\Gamma; \thinspace \Z) \rightarrow \Z/2 \oplus \Z/3 \rightarrow 0,
\\
 0 \rightarrow \Z\oplus \Z/2 \oplus \Z/3 \rightarrow \Homol_1(\Gamma; \thinspace \Z) \rightarrow \Z \rightarrow 0.
\end{cases} \normalsize
$$

Thus, there is ambiguity in the 2- and 3-torsion in $\Homol_2(\Gamma; \thinspace \Z)$, similar to the cases $m=10$ and $m=6$. In order to resolve it, we only need to compute the homology with $\Z/2$- and $\Z/3$-coefficients,
$$
\Homol_q(\Gamma; \thinspace  \Z/2) \cong \scriptsize
\begin{cases}
(\Z/2)^3, & q \in\{1, 2\},\\
(\Z/2)^{2}, & q \ge 3.
\end{cases} \normalsize
\qquad 
\Homol_q(\Gamma; \thinspace  \Z/3) \cong \scriptsize
\begin{cases}
(\Z/3)^3, & q \in\{1, 2\},\\
(\Z/3)^{2}, & q \ge 3.
\end{cases} \normalsize
$$
and then use the Universal Coefficient Theorem to compare. This yields the result:
$$
\Homol_q(\text{PSL}_2(\mathcal{O}_{-15}); \thinspace \Z) \cong
\begin{cases}
\Z^2 \oplus \Z/3 \oplus \Z/2, & q=1,\\
\Z \oplus  \Z/3 \oplus \Z/2, & q=2,\\
\Z/3 \oplus \Z/2, &  q \ge 3.
\end{cases}
$$

\begin{rem}
For $m=15$, the check introduced in remark \ref{check} takes the following form. The abelianization is $\Gamma^{\text{ab}}\cong \langle\overline{AS},\overline{C},\overline{U}: 6\overline{AS}=0 \rangle$. The elements of infinite order $U$ and $C^{-1}$ give the same cycle, which generates $\Homol_1(\Gamma \backslash X)$. However, the element $U^{-1}C^{-1}$ has infinite order as well, and generates a trivial loop in the quotient space.
So it follows that $E_{0,1}^\infty \cong \Z\oplus \Z/2\oplus \Z/3$, generated by $\overline{U^{-1}C^{-1}}$ and $\overline{AS}$. This is consistent with the computation above.
\end{rem}

\bigskip
\section{Appendix: The equivariant retraction}
In this section, we give Fl\"oge's proof of the existence of a retraction $\rho$ from $\widehat{\Hy}$ to the cell complex $X^\bullet$.
We do not show the fact that $\rho$ is $\Gamma$-equivariant, which can be observed since the fibers of $\rho$ are geodesic arcs.\\
\begin{thm}[\cite{FloegePhD}*{theorem 6.6}] \label{equivariant_retraction}
$X$ is a retract of $\widehat{\Hy}$, 
i.\ e. there is a continuous map $\rho:\widehat{\Hy}\to X$ such that $\rho(p)=p$ for all $p\in X$.
\end{thm}
The map $\rho$ is first defined as the orthogonal projection $\pi$ from $\widehat{B}$ to $\partial \widehat{B}$, 
and is then continued to the whole of $\widehat{\Hy}$ by $\Gamma$.
Bianchi \cite{Bianchi}
has shown that a nearly strict fundamental domain for the action of $\Gamma$ on $\Hy$ 
can be chosen in the form of a Euclidean vertical column $D$ inside $B$.
Define $$ \widehat{D} := \{ (z,r) \in \widehat{B} \quad | \quad 0\le \mathrm{Re}(z)\le 1,\quad
0\le \mathrm{Im}(z)\le \sqrt{m}\},$$
and denote by $S$ the set of singular points in $ \widehat{D}$. Finally, $ D := \widehat{D} - S $. \\
\begin{rem}[\cite{FloegePhD}, $D$ is $\Gamma$-normal] \label{neighborhood}
For every $p \in \Hy$, there exists a neighborhood $U$ of $p$ in $\Hy$ such that there are at most finitely many $g \in \Gamma$ with
$gD \cap U \neq \emptyset$.
\end{rem}
We will use the following lemmas to prove theorem \ref{equivariant_retraction}.
\begin{Lem}[\cite{FloegePhD}*{lemma 6.5}] \label{6.5}
For any subset $A \subset D$ which is closed in $\Hy$ and any $p \in \Hy$, 
there exists an open neighborhood $U_p$ of $p$ such that we have for all $g \in \Gamma$: \quad
$gA \cap U_p \neq \emptyset$ if and only if $p \in gA$.
\end{Lem}
\begin{proof}
By remark \ref{neighborhood}, there is a neighborhood $U$ of $p$ in $\Hy$ for which 
$\{ g \in \Gamma \thinspace | \thinspace gD \cap U \neq \emptyset \thinspace \}$ is finite.
In particular, its subset 
$$\Gamma_o := \{ g \in \Gamma \thinspace | \thinspace gA \cap U \neq \emptyset \quad \mathrm{and} \quad p \notin gA \thinspace \}$$ 
is finite. Therefore, $A$ being closed, $\bigcup\limits_{g \in \Gamma_o} gA$ is closed in $\Hy$.
Thus  $U_p := U-(\bigcup\limits_{g \in \Gamma_o} gA)$ is open in $\Hy$ and satisfies to the requested condition. 
\end{proof}
\begin{Lem}[\cite{FloegePhD}*{lemma 6.3}]  \label{6.3}
There is an $\varepsilon_0 > 0$ such that for all singular points $s,s' \in S $,
for all $\varepsilon \le \varepsilon_0 $ and $g \in \Gamma$ we have the following statement: 
\quad
$g \widehat{U}_\varepsilon(s) \cap \widehat{U}_\varepsilon(s') \neq \emptyset$ implies $gs = s'$. 
\end{Lem}
For class number two, as we obtain a fundamental domain for the action of $\Gamma$ on $\widehat{\Hy}$ 
(stricter than~$\widehat{D}$) containing just one singular point, 
this lemma states only that $\Gamma$ acts discontinuously on $\widehat{\Hy}$ 
(with respect to its topology which is finer than the subset topology of $\R^3$);
 and we skip Fl\"oge's proof which is useful for class number three or greater.
\begin{Lem}[\cite{FloegePhD}*{lemma 6.4}]  \label{6.4}
There exists an $\varepsilon_1 > 0$ with the following property: \\
If $\varepsilon \le \varepsilon_1$ and $(z,r) \in \widehat{D}$ with $r < \varepsilon$, 
then there is an $s' \in S$ such that $(z,r) \in \widehat{U}_{2\varepsilon}(s')$.
\end{Lem}
\begin{wrapfigure}{r}{55mm}
\begin{center}
\includegraphics[width=52mm]{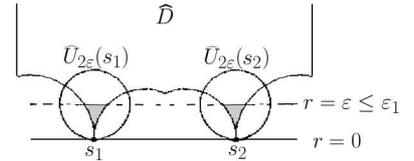} 
\caption{Fl\"oge's sketch} \label{sketch}
\end{center}
\end{wrapfigure}
Fl\"oge draws the sketch of the situation in a vertical half-plane, which we reproduce in figure \ref{sketch} with his kind permission.
He gives only some hints on the proof, which we want to make slightly more explicit here.
\begin{proof}[Sketch of proof]
We consider the Euclidean geometry of the upper-half space model for $\widehat{\Hy}$ and write coordinates in $\C \times \R^{\ge 0}$ .
Denote by $\varepsilon_1$ the ``height of the lowest non-singular vertex'', 
more precisely the minimum of the values $r>0$ occuring as the real coordinate of the non-singular vertices $(z,r) \in \Hy$ of the fundamental domain $\rho(\widehat{D})$ for $\Gamma$. 
Then $\{(z,r) \in \widehat{D} \thinspace | \thinspace r < \varepsilon_1 \}$ 
consists of one connected component for each singular point $s' \in S$. 
We will denote by $\widehat{D}_{s'}$ the connected component containing $s'$.
Now fix $s' \in S$. 
There are finitely many hemispheres limiting $\widehat{D}$ from below and touching $s'$.
We will consider the situation in a vertical half-plane containing $s'$.
The most critical  vertical half-planes for our assertion contain the intersection arc of two such hemispheres,
because the other vertical half-planes contain circle segments of $\partial \widehat{D}$ of greater radius.
The intersection of two non-identical Euclidean 2-spheres which have more than one point in common,
is a circle with center on the line segment connecting the two 2-sphere centers.
Thus the intersection of the two hemispheres mentioned above is a semicircle with center in the plane $r=0$ .
Denote by $\zeta$ the radius of this semicircle. 
Then $\varepsilon_1 \le \zeta $, because an edge of our fundamental domain, connecting $s'$ with a non-singular vertex, lies on this semicircle.
Now it is easy to see that $\widehat{D}_{s'}$  is a subset of the truncated cone 
obtained as the convex envelope of $s'$ and the horizontal disk with radius $\zeta$ and center $(s',\zeta)$.
We conclude that for all  $\varepsilon < \varepsilon_1$, $\varepsilon >0$, the set $\{(z,r) \in \widehat{D}_{s'} \thinspace | \thinspace r < \varepsilon \}$ 
is a subset of the horoball $\widehat{U}_{2\varepsilon}(s')$.
So we have seen that $\varepsilon_1$ has the property claimed in the lemma.
\end{proof}
\begin{proof}[Proof of theorem \ref{equivariant_retraction}]
For any $(z,r)\in \widehat{D}$ there is a unique $r_z\ge 0$ such that $(z,r_z)\in \widehat{D}\cap \partial \widehat{B} =: \widehat{G}$,
 in fact $r_z=\text{min }\{r':(z,r')\in \widehat{D} \}$. We can thus define the map $\pi: \widehat{D} \to \widehat{G}$ by $\pi(z,r):=(z,r_z)$. The map $\pi$ is continuous with respect to the subset topology of $\R^3$, and by \cite{FloegePhD}*{corollary 5.10} also with respect to the topology of $\widehat{\Hy}$. Furthermore, we have $\pi(p)=p$ for all $p\in \widehat{G}$.
We now extend $\pi$ to a map $\rho:\widehat{\Hy}\to X$ as follows. 
Because of \mbox{$\left\{\left(\begin{smallmatrix}1&b\\&1\end{smallmatrix}\right):b\in R\right\}\cdot \widehat{D}=\widehat{G}$,} 
we find for any $p\in\widehat{\Hy}$ a $\gamma\in\Gamma$ such that $\gamma(p)\in \widehat{D}$. 
We set \mbox{$\rho(p):=\gamma^{-1}\circ \pi\circ\gamma(p)$.}
 In order to show that this makes sense, we have to show that $p\in \gamma^{-1}\widehat{D}\cap \xi^{-1}\widehat{D}$ implies 
\mbox{$\gamma^{-1}\circ\pi\circ\gamma(p)=\xi^{-1}\circ\pi\circ\xi(p)$,}
 where \mbox{$\gamma,\xi\in\Gamma$.}
We have $\xi(p)\in \xi\gamma^{-1}\widehat{D}\cap \widehat{D}$, then 
\mbox{$\gamma\xi^{-1}(\xi(p))=\gamma(p)\in \widehat{D}\cap \gamma\xi^{-1}\widehat{D}$,}
 and either $\xi(p),\gamma(p)$ are both from $\widehat{G}$, or both from $\widehat{D}\cap B^\circ$. In the first case, it immediately follows that $\gamma^{-1}\circ\pi\circ\gamma(p)=\xi^{-1}\circ\pi\circ\xi(p)=p$, and $\xi^{-1}\circ \xi (p)=p$. In the second case, we have by \cite{FloegePhD}*{lemma 3.4} that if $\gamma\xi^{-1}=\left(\begin{smallmatrix}
a&b\\c&d
\end{smallmatrix}
\right)$, the entry $c$ must vanish. So $\gamma\xi^{-1}$ is the product $\left(\begin{smallmatrix}
a&0\\0&d
\end{smallmatrix}\right)
\left(\begin{smallmatrix}
1&db\\0&1.
\end{smallmatrix}\right)$. Both of the latter two matrices commute with $\pi$ since any such element $\zeta$ satisfies $\zeta(\partial \widehat{B})=\partial \widehat{B}$, and $\zeta$ maps vertical half-lines to vertical half-lines.\\
So we have $(\gamma\xi^{-1}\circ\pi\circ \xi\gamma^{-1})p'=\pi p'$ for all $p'\in \widehat{D}$ with $\xi\gamma^{-1}p'\in \widehat{D}$, and then it follows that 
$$\xi^{-1}\circ\pi\circ\xi(p)=\gamma\in\gamma(\xi^{-1}\circ\pi\circ\xi)\gamma^{-1}\gamma(p)=\gamma^{-1}\circ\pi\circ\gamma(p)=\gamma^{-1}\circ\pi\circ\gamma(p).$$
 Thus, $\rho$ is well-defined. Furthermore, $\pi(p)=p$ for all $p\in \widehat{G}$ implies $\rho(p)=p$ for all $p\in X.$
It remains to show that $\rho$ is continuous at any $p\in\widehat{\Hy}$.
\paragraph{1st case.} In the case $p\in\Hy$, by lemma \ref{6.5}, $p$ has an open neighborhood $U_p$ such that: for any $\gamma\in\Gamma$, we have $\gamma U_p\cap D\ne \emptyset\iff \gamma(p)\in D$. Furthermore, the set $\{\gamma\in\Gamma:\gamma(p)\in D\}$ is finite \cite{FloegePhD}*{remark 3.6}, say $\gamma_1,\dots,\gamma_n$.
Let now $V$ be an open neighborhood of $\rho(p)$. Because of the continuity of all $\gamma_i,\gamma_i^{-1}$ and the continuity of $\pi:\widehat{D}\to \widehat{G}$, there exist neighborhoods $U_i$ of $p$ such that $\gamma_i^{-1}\circ\pi\circ \gamma_i(U_i)\subset V$. Note that for all $\gamma_i$ we have $\gamma_i^{-1}\circ\pi\circ\gamma_i(p)=\rho(p)$.
Setting $U:=U_p\cap (\bigcap_{i=1}^n U_i)$, we have $\rho(U)\subset V$, i.~e. $\rho$ is continuous at the point $p$.
\paragraph{2nd case.}
In the case $p\in\widehat{\Hy}\cap\C$, let $\epsilon_0,\epsilon_1$ and $\epsilon_s$ for $s\in S$ be positive real numbers as 
in lemma \ref{6.3}, lemma \ref{6.4} and \cite{FloegePhD}*{lemma 5.9}; 
and let $\epsilon > 0$ be less than the minimum of $\frac{\epsilon_0}{2},\epsilon_1, \epsilon_s $ for $s\in S$. 
Because of 
\mbox{$\left\{\left(\begin{smallmatrix}1&b\\&1\end{smallmatrix}\right):b\in R\right\}\cdot \widehat{D}=\widehat{G}$,}
 there exist $s\in S$, 
$\xi=\left(\begin{smallmatrix}a&b \\c&d\end{smallmatrix} \right)$
 such that $\xi s=p$ and by \cite{FloegePhD}*{remark 5.5(a)},
 we have $\xi\widehat{U}_\epsilon(s)=\widehat{U}_{\frac{\epsilon}{\left|cs-d\right|^2}}(p)$. 
Let us now show that $\rho(\widehat{U}_{\frac{\epsilon}{\left|cs-d\right|^2}}(p))\subset \widehat{U}_{2\epsilon}(p)$.
Let $p'\in\widehat{U}_{\frac{\epsilon}{\left|cs-d\right|^2}}(p)$, and let $\gamma\in\Gamma$ with $\gamma p'\in \widehat{D}$. 
Then \mbox{$\rho(p')=\gamma^{-1}\circ\pi \circ \gamma(p')$.} 
By \cite{FloegePhD}*{remark 5.5(b)}, applied to $s$ and $\gamma\xi$ it follows that 
\mbox{$\gamma p'=\gamma\xi(\xi^{-1}p')\in\widehat{U}_\epsilon(\gamma\xi s)=\widehat{U}_\epsilon (\gamma p)$,}
 and by \cite{FloegePhD}*{remark 5.6} all conditions of lemma \ref{6.4} are satisfied.
 So there is an $s'\in S$ such that $\gamma p'\in\widehat{U}_{2\epsilon}(s')$. 
This means that $\gamma\xi(\widehat{U}_{2\epsilon}(s))\cap \widehat{U}_{2\epsilon}(s')\ne\emptyset$, 
and by lemma \ref{6.3} it follows that $s'=\gamma\xi s=\gamma p$.
Let us now consider $\gamma p'$ again. 
\\
 Since $\gamma p'\in\widehat{U}_{\epsilon}(\gamma p))=\widehat{U}_{\epsilon}(s')=U_\epsilon(s')$ and $\pi(U_{\epsilon}(s'))\subset U_{\epsilon}(s')$;
 and by \cite{FloegePhD}*{lemma 5.9} we have $U_\epsilon(s')\cap \widehat{B}\subset \widehat{U}_{2\epsilon}(s')$. 
So $\pi\circ\gamma p'\in\widehat{U}_{2\epsilon}(s').$ By  \cite{FloegePhD}*{remark 5.5(b)} it finally follows that
$$
\rho(p')=\gamma^{-1}\circ\pi\circ\gamma p'
\in\gamma^{-1}\widehat{U}_{2\epsilon}(s')\subset \widehat{U}_{2\epsilon}(\gamma^{-1}s')=\widehat{U}_{2\epsilon}(p),
$$
and we are done.
\end{proof}

\end{document}